\documentclass[reqno,12pt]{amsart}
\usepackage{amscd,amsfonts,mathrsfs,amsthm,amssymb, amsmath}
\usepackage{enumerate}
\usepackage{multicol}        			
\usepackage{color}           			
\usepackage{array}
\usepackage{ae}
\usepackage{accents}
\usepackage{epsfig}
\usepackage{pstricks, pst-node, pst-plot, pst-math}
\usepackage[e]{esvect}
\usepackage[alphabetic, sorted-cites]{amsrefs}

\psset{unit=1em}
\newcommand{\pg}{\mathfrak{p}}
\newcommand{\R}{\mathbb{R}}

\def\comment#1 {{\color{red}(Comment: #1) }}
\def\smi{\hspace{-3pt}-\hspace{-3pt}}
\def\p{\operatorname{p}}
\def\H{\vec{H}}
\def\A{\vec{A}}
\def\v{\operatorname{v}}

\def\real     #1{{\mathbb R^{#1}}}

\def\natural  #1{{\mathbb N^{#1}}}

\def\mreg    {M_{\operatorname{reg}}}

\def\gind {\operatorname{g}}
\def\dt       {\frac{d}{dt}\,}

\newtheorem{theorem}{Theorem}[section]
\newtheorem{mythm}{Theorem}

\newtheorem{lemma}[theorem]{Lemma}
\newtheorem*{thma}{Theorem A}
\newtheorem*{thmb}{Theorem B}
\newtheorem*{thmc}{Theorem C}
\newtheorem*{thmd}{Theorem D}
\newtheorem{corollary}[theorem]{Corollary}

\newtheorem{definition}[theorem]{Definition}
\theoremstyle{definition}
\newtheorem{remark}[theorem]{Remark}

\def\pproof#1{\@ifnextchar[\opargproof
{\opargproof[\it Proof of #1.]}}
\def\opargproof[#1]{\par\noindent {\bf #1 }}

\makeatother

\numberwithin{equation}{section}

\begin{document}

\title[Translating solitons]{On the topology of translating solitons of the mean curvature flow}
\author[F. Mart\'in]{\textsc{F. Mart\'in}}
\author[A. Savas-Halilaj]{\textsc{A. Savas-Halilaj}}
\author[K. Smoczyk]{\textsc{K. Smoczyk}}
\address{Francisco Mart\'in\newline
Departmento de Geometr\'ia y Topolog\'ia\newline
Universidad de Granada\newline
18071 Granada, Spain\newline
{\sl E-mail address:} {\bf fmartin@ugr.es}
}
\address{Andreas Savas-Halilaj\newline
Institut f\"ur Differentialgeometrie\newline
Leibniz Universit\"at Hannover\newline
Welfengarten 1\newline
30167 Hannover, Germany\newline
{\sl E-mail address:} {\bf savasha@math.uni-hannover.de}
}
\address{Knut Smoczyk\newline
Institut f\"ur Differentialgeometrie\newline
and Riemann Center for Geometry and Physics\newline
Leibniz Universit\"at Hannover\newline
Welfengarten 1\newline
30167 Hannover, Germany\newline
{\sl E-mail address:} {\bf smoczyk@math.uni-hannover.de}
}

\date{}
\thanks{The first author is partially supported by MICINN-FEDER grant no. MTM2011-22547}
\thanks{The second author is is supported by the grant of E$\Sigma\Pi$A : PE1-417}

\begin{abstract}
In the present article we obtain classification results and topological obstructions for the existence of
translating solitons of the mean curvature flow in euclidean space.
\end{abstract}

\maketitle

\section{Introduction}
An oriented smooth hypersurface $f:M^m\to\real{m+1}$  is called {\it translating soliton}
(or a {\it translator} for short) of the
mean curvature flow if its mean curvature vector
field $\H$ satisfies
\begin{equation}\label{trans 1}
\H={\v}^\perp,
\end{equation}
where ${\v}\in\real{m+1}$ is a fixed unit length vector and $\v^\perp$ stands for the orthogonal
projection of ${\v}$ onto the normal bundle of the immersion $f$.
Translating solitons are important in the singularity theory of the mean curvature flow since they often occur
as Type-II singularities. On the other hand they also form interesting examples of precise solutions of the flow
since the smooth family of immersions $F:M^m\times\real{}\to\real{m+1}$,
$F(x,t):=f(x)+t{\v}$,
evolves, up to some tangential diffeomorphisms, by its mean curvature.
If one chooses a smooth unit length normal vector field $\xi$ along $f$, then equation
(\ref{trans 1}) may be expressed in terms of scalar quantities. More precisely, equation
(\ref{trans 1}) is equivalent to
\begin{equation}\label{trans 1b}
H:=-\langle\H,\xi\rangle=-\langle {\v},\xi\rangle,
\end{equation}
where here $H$ is the {\it scalar mean curvature} of $f$.
Since equation (\ref{trans 1}) is invariant under isometries one may, without loss of generality,
always assume that the velocity vector ${\v}$ is given by
${\v}=\operatorname{e}_{m+1},$
where $\{\operatorname{e}_1,\dots,\operatorname{e}_{m+1}\}$ denotes the standard orthonormal basis of $\real{m+1}$.

Translating solitons of the euclidean space $\real{m+1}$ are closely related to minimal hypersurfaces. In fact, translators can be regarded
as minimal hypersurfaces of $(\real{m+1},\operatorname{G})$ where $\operatorname{G}$ is a Riemannian metric conformal
to the usual inner product of $\real{m+1}$ (for more details see Section \ref{alexandrov}). However, at present there is no general method
to construct examples of translating solitons. Even in the $2$-dimensional case there is no Weierstra{\ss} type representation known to exist
for translators, like there is for minimal surfaces in $\real{3}$. Moreover, although it is
believed that there exists an abundance of translators, there are only a very few available examples of complete translating solitons in the euclidean space
$\real{m+1}.$ For instance, any minimal hypersurface of $\real{m+1}$ tangent to the translating direction $\v$ is a translating soliton (however, in this case $H\equiv 0$
implies that the translator actually does not move at all). The euclidean
product $\Gamma\times\real{m-1}$, where $\Gamma$ is the {\it grim reaper} in $\real{2}$ represented by the
immersion $f:(-\pi/2,\pi/2)\to\real{2}$ given by
$$f(x)=-\log\cos x,$$
gives rise again to a translating soliton. Any translating soliton that up to a rigid motion coincides with $\Gamma\times\real{m-1}$ will be called a {\it grim hyperplane}.
Another
way to construct complete translating solitons is by rotating special curves around the translating axis.
From this procedure one gets the
rotational symmetric {\it translating paraboloid} and the {\it translating catenoids} which are unique
up to rigid motions (see the examples in subsection \ref{examples}d). We would like to point out that
it is not known if there are examples of complete translating solitons with finite non-zero genus.

Our goal is to classify, under suitable conditions, translating solitons and to
obtain topological obstructions for the existence of translating solitons with finite non-zero genus in the euclidean space $\real{3}$.
In the second section, after setting up the notation and computing the basic equations that translators satisfy, we give the following
characterization of the grim hyperplanes.

\begin{thma}
Let $f:M^m\to\real{m+1}$ be a translating soliton which is not a minimal hypersurface.
Then $f(M^m)$ is a grim hyperplane if and only if the function $|A|^2H^{-2}$ attains a local maximum
on the open set $M^m-\{H=0\}$.
\end{thma}

As an immediate consequence of the above theorem we prove that a translating soliton with zero scalar
curvature either coincides with a grim hyperplane or with a minimal hypersurface tangential to the translating direction.

In the third section, we prove a uniqueness theorem for complete embedded translating solitons with a single end that are
asymptotic to a translating paraboloid. Our proof relies heavily on the Alexandrov's reflection principle. More precisely
we show the following.

\begin{thmb}\label{reflection}
Let $f:M^m\to\real{m+1}$ be a complete embedded translating soliton of the mean curvature flow
with finite genus and a single end that is smoothly asymptotic to a translating paraboloid. Then
$M=f(M^m)$ is a translating paraboloid.
\end{thmb}

The last section is devoted to translating solitons in $\real{3}$. We focus on the study of the
distribution of the Gau{\ss} map of a complete translating soliton. In particular we investigate
how the Gau{\ss} image affects the genus of a complete translating soliton in $\real{3}$.
We prove that within the class of translating solitons for which the Gau{\ss} map omits the north
pole (or more generally the direction of translation) it holds that a translator is mean convex, if
it is mean convex outside a compact subset. More precisely, we show the following.

\begin{thmc}
Let $f:M^2\to\real{3}$ be a translating soliton whose mean curvature satisfies $H>-1$.
Suppose that $H\ge 0$ outside a compact subset of $M$. Then either $M=f(M^2)$ is part of
a flat plane or $H>0$ on all of $M^2$. If, in addition, $M$ is
properly embedded, then it is a graph and so it has genus zero.
\end{thmc}

We would like to point out here that from the equation (\ref{trans 1b}) the scalar mean curvature $H$
of a translating soliton always satisfies the inequality
$$-1\le H=-\langle\v,\xi\rangle\le 1.$$
Hence, the assumption $H>-1$ means that the Gau{\ss} map $\xi$ is omitting the north pole 
of 
$\mathbb{S}^2$. 

At this point we would like to mention that it is perhaps true that {\it any
complete translating soliton in $\real{3}$ whose Gau{\ss} map omits the north pole
must have genus zero}. In the following theorem we give a partial answer to this
question.

\begin{thmd}
Suppose that $\Sigma_{g}$ is a compact Riemann surface of genus $g$ and let $\{p_1,\dots,p_k\}$ be
distinct points on $\Sigma_g$. Suppose further that $f:\Sigma_{g,k}:=\Sigma_g\smi\,\{p_1,\dots,p_k\}\to\real{3}$
is a complete translating soliton that satisfies the following two conditions:
\begin{enumerate}[\rm(a)]
\item each end is either bounded from above or from below,
\smallskip
\item the mean curvature satisfies $H>-1$ everywhere and
$$\limsup_{x\to p_j}H^2(x)\le 1-\varepsilon$$
for some positive constant number $\varepsilon$ and for
any of the punctures $p_j\in\{p_1,\dots,p_k\}$. 
\end{enumerate}
Then $g\le 1$. If $f$ is an embedding, then $g=0$ and thus the Riemann surface $\Sigma_{g,k}$ is a planar domain.
\end{thmd}

As mentioned earlier, it is not known, if there exists any complete translator with genus 1.
\section{Translating solitons}

\subsection{Local formulas}
In this subsection we will derive and summarize the most relevant equations related to translators.
Let $f:M^m\to\real{m+1}$ be an immersion and $\gind=f^*\langle\cdot,\cdot\rangle$ the induced metric. Denote
by $D$ the Levi-Civita
connection of $\real{m+1}$ and by $\nabla$ the Levi-Civita connection of $\gind$. The {\it second fundamental
form} $\A$ of $f$ is
$${\A}(v,w):=D_{\operatorname{d}{\hspace{-2pt}}f(v)}\operatorname{d}{\hspace{-2pt}}f(w)-
\operatorname{d}{\hspace{-4pt}}f\big(\nabla_vw\big),$$
where $v,w$ are tangent vectors on $M$. The {\it mean curvature
vector field} ${\H}$ of $f$ is defined by
$${\H}:=\operatorname{trace}_{\gind}{\A}.$$
Let now $\xi$ be a local unit vector field normal along $f$. The symmetric bilinear form $A$ given by
$$A(v,w):=-\langle {\A}(v,w),\xi\rangle,$$
where $v,w\in TM^m$, is called the {\it scalar second fundamental from} of $f$.
The {\it scalar mean curvature} $H$ is defined as the
trace of $A$ with respect to $\gind$.
Suppose now that $f$ is a translating soliton, that is
$${\H}={\v}^{\perp}$$
where ${\v}=(0,\dots,0,1)\in\real{m+1}$ and where ${\v}^\perp$ is the orthogonal projection of ${\v}$
onto the normal bundle of $f$. The orthogonal projection of ${\v}$ onto the tangent bundle
of $f$ will be denoted by ${\v}^{\top}$.
Let us introduce the height function $u:M^m\to\real{}$, given by
\begin{equation}\label{trans 2}
u:=\langle f,{\v}\rangle.
\end{equation}
In the next lemma we give some important relations between the mean curvature $H$ and the height function $u$.
\begin{lemma}\label{lemm rel}
The following equations hold on any translating hypersurface in $\real{m+1}$.
\begin{enumerate}[\rm(\rm a)]
\item $\nabla u={\v}^{\top},$\\[-10pt]
\item $|\nabla u|^2=1-H^2,$\\[-10pt]
\item $\nabla^2u=HA,$\\[-10pt]
\item $\Delta u +|\nabla u|^2-1=0,$\\[-10pt]
\item $\langle\nabla H,\cdot\,\rangle=-A(\nabla u,\cdot\,),$\\[-10pt]
\item $\Delta H+H|A|^2+\langle\nabla H,\nabla u\rangle=0,$\\[-10pt]
\item $\operatorname{Ric}(\nabla u,\nabla u)=-|\nabla H|^2-H\langle\nabla H,\nabla u\rangle,$\\[-10pt]
\item $\Delta|A|^2-2|\nabla A|^2+\langle\nabla|A|^2,\nabla u\rangle+2|A|^4=0.$
\end{enumerate}
\end{lemma}
\begin{proof}
Let $\{e_1,\dots,e_m\}$ be an orthonormal frame defined on an open neighborhood of $M^m$.
\begin{enumerate}[\rm(a)]
\item
Differentiating $u$ with respect to $e_i$ we get
$$e_iu=\langle\operatorname{d}\hspace{-2pt}f(e_i),{\v}\rangle.$$
Therefore,
\begin{equation*}
\nabla u={\v}^\top.
\end{equation*}
\item
Since ${\v}$ has unit length, we obtain the crucial identity
\begin{equation*}
1=|{\v}|^2=|{\v}^\perp|^2+|{\v}^\top|^2=H^2+|\nabla u|^2.
\end{equation*}
\item
Differentiating $\nabla u$ once more, we deduce that
\begin{eqnarray*}
\nabla^{2}u(e_i,e_j)&=&e_ie_j u-\langle\nabla u,\nabla_{e_i}e_j\rangle\\
&=&e_i\langle\operatorname{d}\hspace{-2pt}f(e_j),{\v}\rangle-\langle\operatorname{d}\hspace{-2pt}f(\nabla_{e_i}e_j),{\v}\rangle\\
&=&\langle {\A}(e_i,e_j),{\v}\rangle\\
&=&HA(e_i,e_j).
\end{eqnarray*}
\item
From the last formula giving the Hessian of $u$, we see that
$$\Delta u=H^2=1-|\nabla u|^2.$$
\item
Differentiating $H$ with respect to the direction $e_i$, we get
\begin{eqnarray*}
\langle\nabla H, e_i\rangle&=&-e_i\langle {\v},\xi\rangle
=-\langle {\v},\operatorname{d}\hspace{-2pt}\xi(e_i)\rangle
=-\langle{\v}^{\top},\operatorname{d}\hspace{-2pt}\xi(e_i)\rangle\\
&=&-A(\nabla u,e_i).
\end{eqnarray*}
\item
Differentiating once more the gradient of $H$ and using the Codazzi equation, we have
\begin{eqnarray*}
\nabla^{2}H(e_i,e_j)&=&-\sum_{k=1}^{m}(\nabla_{e_k}A)(e_i,e_j)e_{k}u\\
&&-\sum_{k=1}^{m}A(e_k,e_j)\nabla^{2}u(e_i,e_k)\\
&=&-\sum_{k=1}^{m}(\nabla_{e_k}A)(e_i,e_j)e_{k}u\\
&&-H\sum_{k=1}^{m}A(e_k,e_j)A(e_i,e_k).
\end{eqnarray*}
Therefore,
$$\Delta H=-\sum_{k=1}^{m}(e_kH)(e_ku)-H|A|^2=-\langle\nabla H,\nabla u\rangle-H|A|^2.$$
\item
From the relation (e), we get
$$|\nabla H|^2=A^{[2]}(\nabla u,\nabla u)$$
where $A^{[2]}$ is the symmetric $2$-tensor given by
$$A^{[2]}(v,w)=\sum_{i=1}^mA(v,e_i)A(e_i,w),$$
for any $v,w\in TM^m$.
Again from the relation (e), we have
$$H\langle\nabla H,\nabla u\rangle=-HA(\nabla u,\nabla u).$$
Consequently,
$$\big(HA-A^{[2]}\big)(\nabla u,\nabla u)=-|\nabla H|^2-H\langle\nabla H,\nabla u\rangle.$$
By Gau\ss' equation, the left hand side of the above equation equals the Ricci curvature of the induced metric applied to $\nabla u$. Therefore,
\begin{equation*}
\operatorname{Ric}(\nabla u,\nabla u)=-|\nabla H|^2-H\langle\nabla H,\nabla u\rangle.
\end{equation*}
\item
From Simons' formula \cite{simons}, we have that
\begin{eqnarray*}
\frac{1}{2}\Delta|A|^2&=&|\nabla A|^2-|A|^4+\sum_{i,j=1}^mA(e_i,e_j)\nabla^2H(e_i,e_j)\\
&&+H\hspace{-4pt}\sum_{i,j,k=1}^mA(e_i,e_j)A(e_j,e_k)A(e_k,e_i).
\end{eqnarray*}
Bearing in mind the formula which relates the Hessian of $H$ with $A$ and $u$, we get the desired formula.
\end{enumerate}
This completes the proof of the lemma.
\end{proof}

\begin{remark}
From Lemma \ref{lemm rel} (d) it follows that the height function $u$ does not admit any local maxima.
In particular the manifold $M^m$ cannot be compact, a fact that intuitively is clear.
Moreover,  Lemma \ref{lemm rel} (b) and (e) imply that
the critical sets $\operatorname{Crit}(H)$, $\operatorname{Crit}(u)$ of the mean curvature and the height
function on a translating soliton satisfy
\begin{eqnarray*}
\operatorname{Crit}(u)&=&\{x\in M:H^2(x)=1\}\\
\operatorname{Crit}(u)&\subset&\operatorname{Crit}(H),\\
\operatorname{Crit}(H)\smi\operatorname{Crit}(u)&\subset&\mreg\cap\{x\in M:\det A=0\},
\end{eqnarray*}
where $\mreg:=M^m\smi\operatorname{Crit}(u)$
denotes the \emph{regular} part of $M^m$. 
\end{remark}

\subsection{Examples}\label{examples}
We will expose here some examples of translators in the euclidean space.
\begin{enumerate}[(a)]
\item
All solutions $f:(-\pi/2,\pi/2)\to\real{2}$ of (\ref{trans 1}) in the plane
$\real{2}$ are of the form
$f(x)=(x,c-\log\cos x),$
\begin{figure}[h]
\includegraphics[scale=.06]{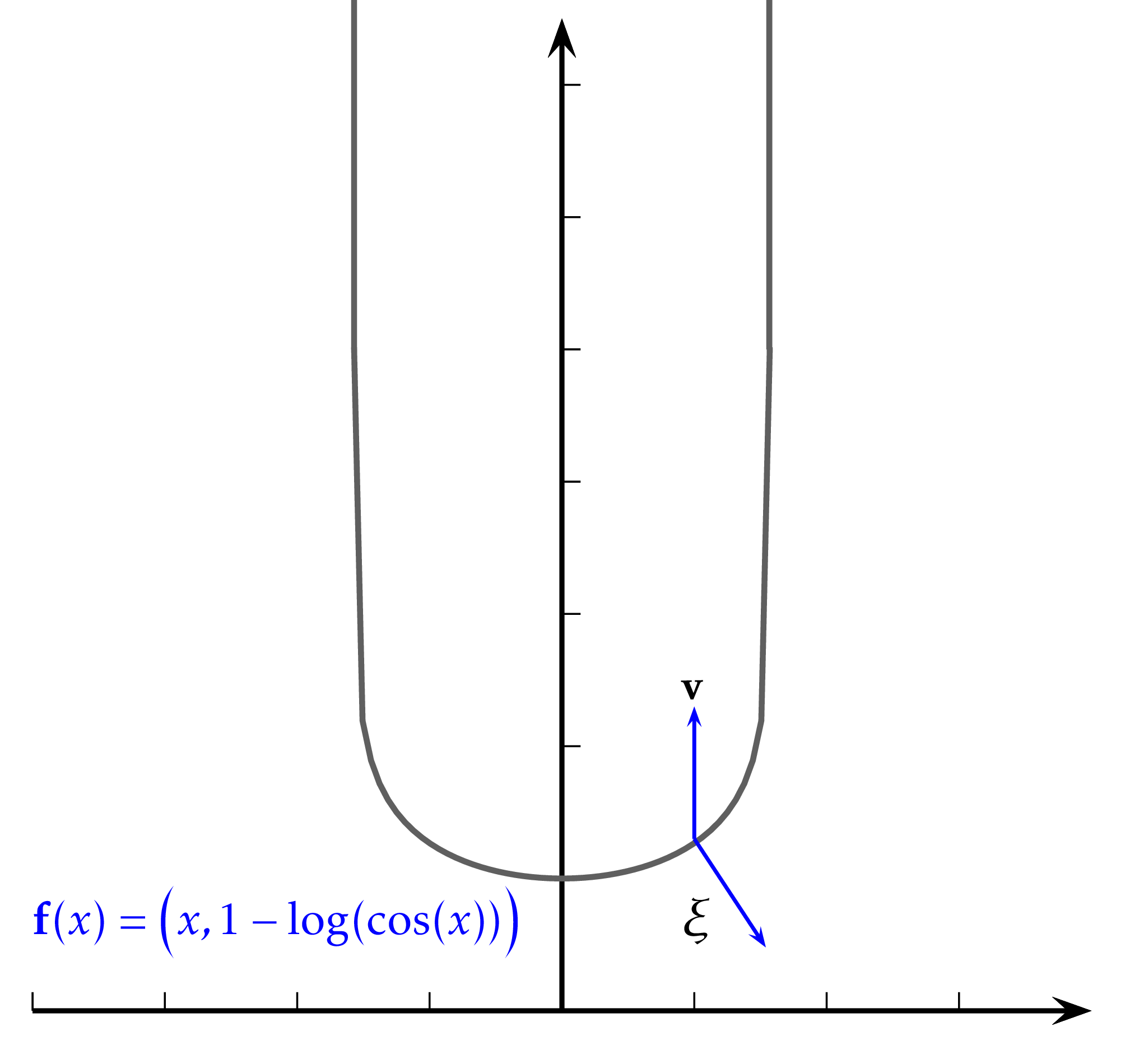}\caption{Grim reaper}\label{grimreaper}
\end{figure}
where $c$ is a real constant. The curve $f$ is known as the {\it grim reaper}.
These curves are geodesics of the Riemannian metric
$$\operatorname{G}_p:=e^{2\langle p,\v\rangle}\langle\cdot\,,\cdot\rangle,$$
on $\real{2}$.
The
\textit{grim hyperplane} in $\real{m+1}$ is the orthogonal product of the grim reaper
\begin{figure}[h]
\includegraphics[scale=.08]{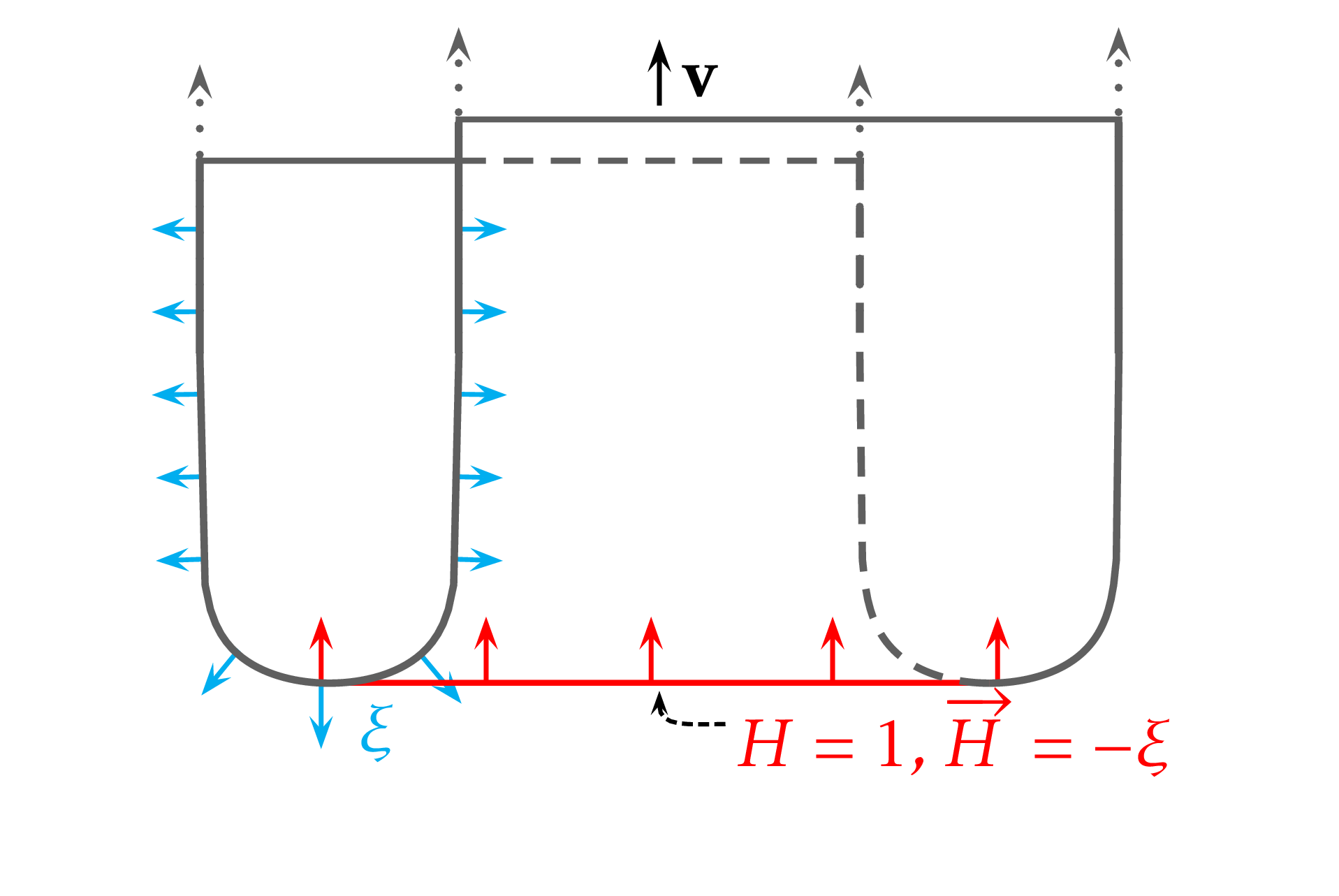}\caption{Grim hyperplane}\label{grimreaperplane}
\end{figure}
with an euclidean factor $\real{m-1}$. These examples are mean convex.
In fact, these examples have only one non-zero principal curvature.\\
\item
Suppose that $f$ is a  translator which is also minimal. Then ${\v}$ must be tangential to the translator.
Consequently, the only 
\begin{figure}[h]\includegraphics[scale=.08]{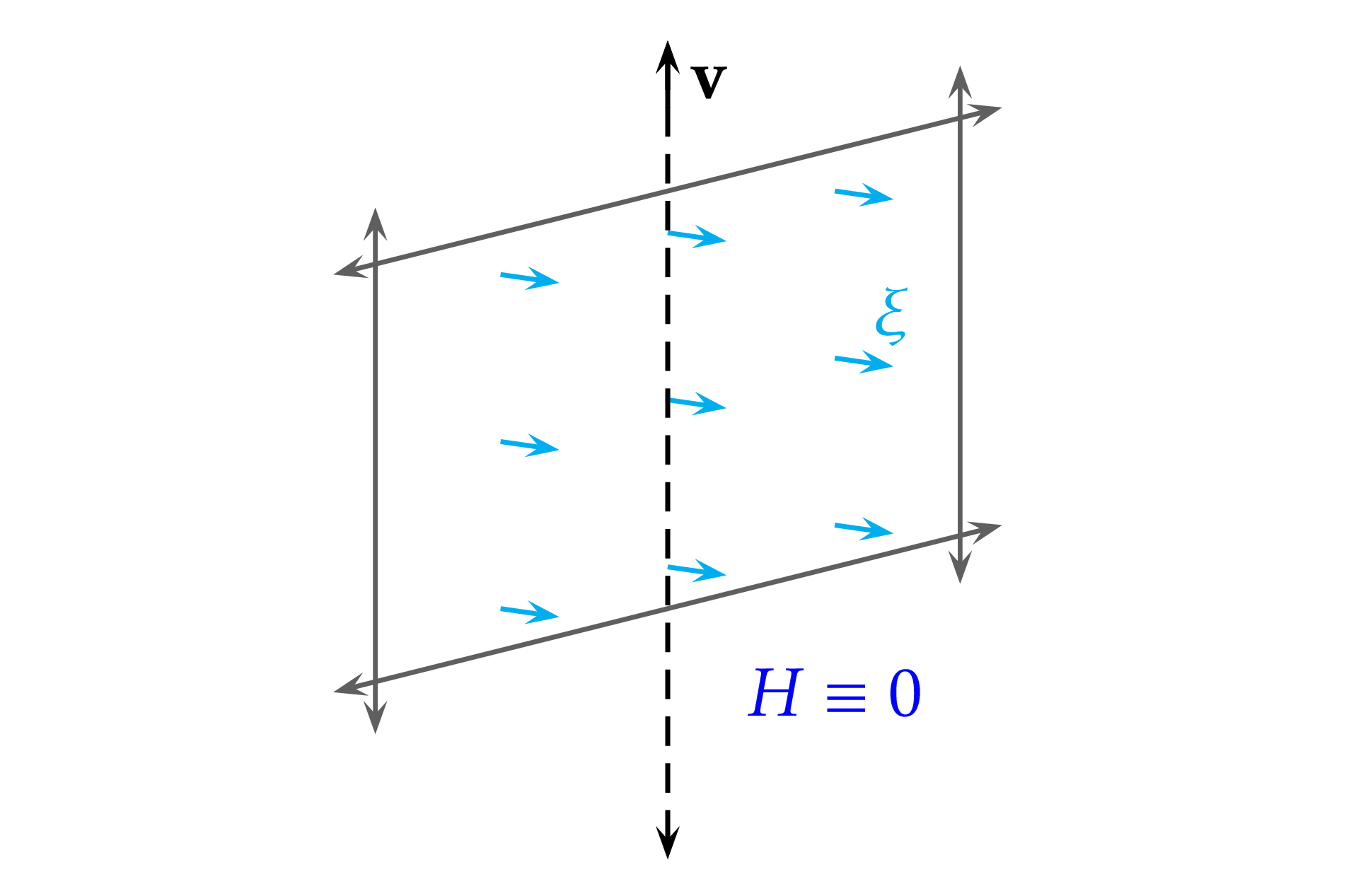}\caption{Flat plane tangent to ${\v}$}\label{plane}\end{figure}
minimal translating surfaces in $\real{3}$ are the flat planes which are tangential to the vector
${\v}$.\\
\item
Altschuler and Wu \cite{aw} evolved graphs by mean curvature flow defined over compact 
convex domains $\Omega$ in $\real{2}$ with prescribed contact angle to the boundary $\partial\Omega$.
They were able to prove that solutions converge to translating solitons that are neither convex nor
rotationally symmetric. Moreover, they showed the existence of complete, rotationally
symmetric translators.\\
\item
If one assumes rotational symmetry around the $\v$ axis, then the equation describing the translators
reduces to an ODE. As was shown by Clutterbuck, Schn\"urer and Schulze \cite{css}, such
complete rotationally symmetric translating solitons coincide (up to translations) either with
\begin{figure}[h]\includegraphics[scale=.08]{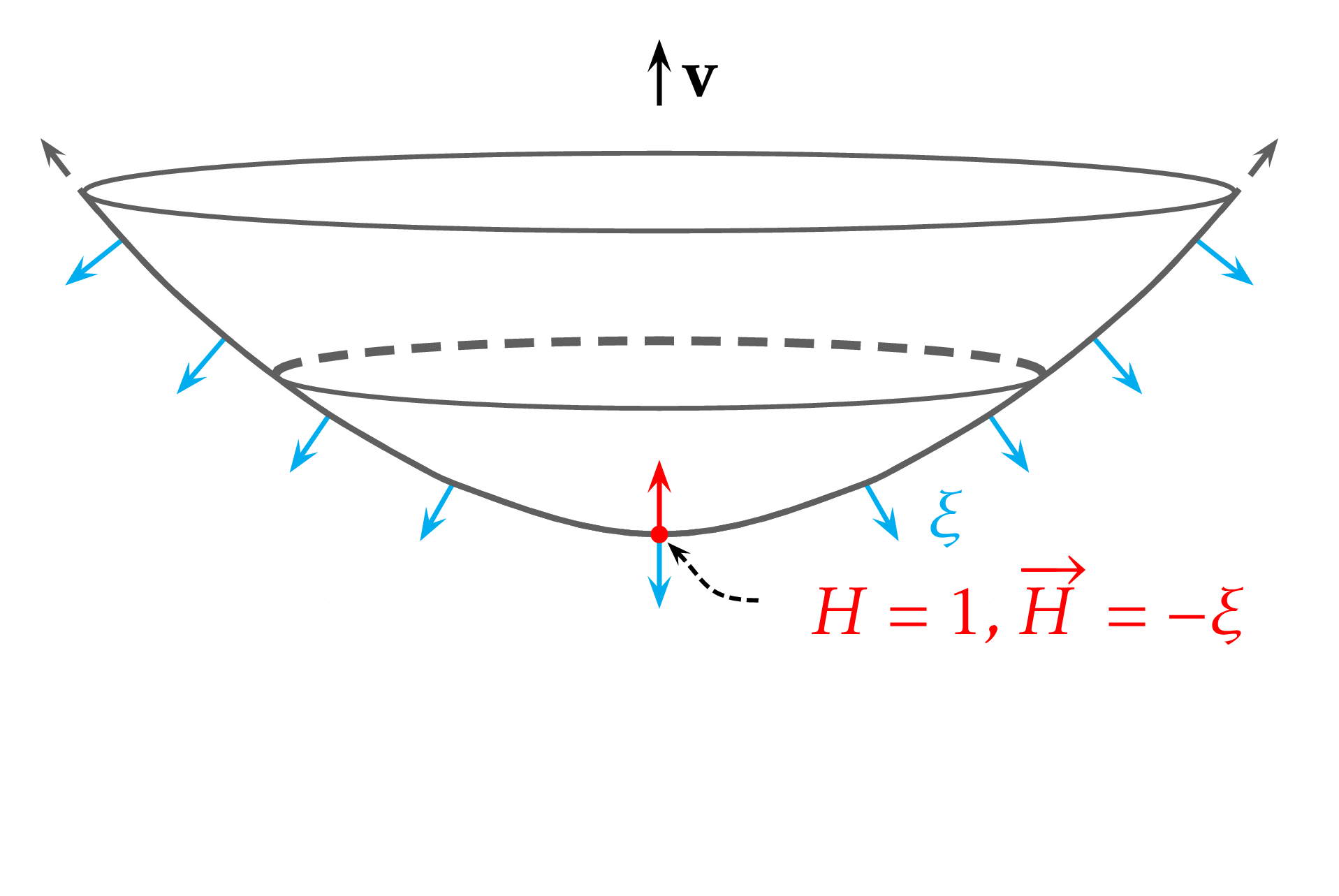}\includegraphics[scale=.08]{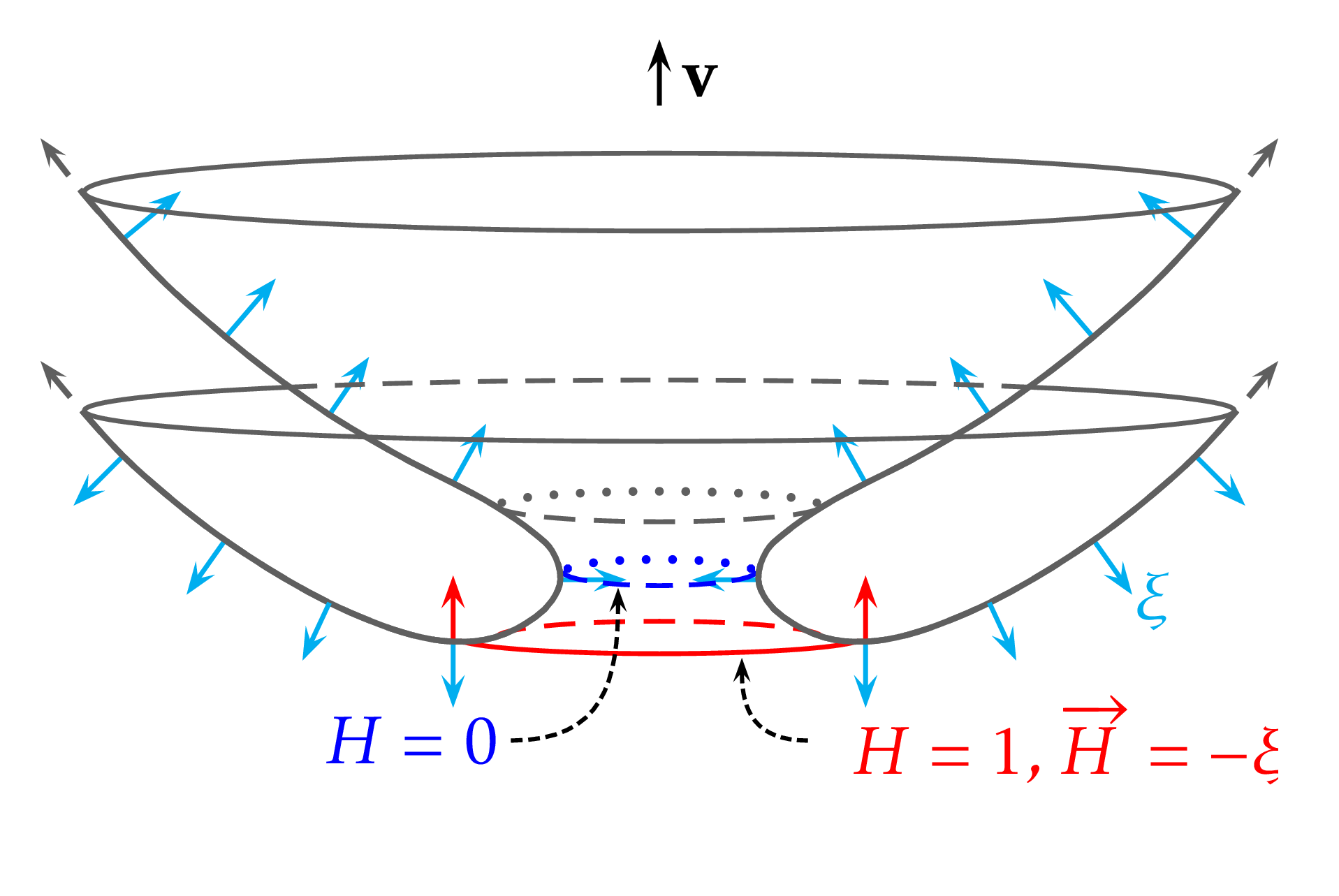}\caption{Translating paraboloid and translating catenoid}\label{translatingparaboloid}\end{figure}
the rotational symmetric {\it translating paraboloid} or with the rotationally symmetric
{\it translating catenoids} which can be seen as the desingularization
of two paraboloids connected by a small neck of some radius.\\
\item
Halldorsson \cite{halldorsson} proved the existence
of helicoidal type translators. Nguyen \cite{nguyen1,nguyen} desingularized  the intersection of a grim reaper
and a plane by a Scherk's minimal surface in $\real{3}$, and obtained a complete embedded translator of infinite genus.
It is not known if there are complete translating solitons with finite non-trivial topology.\\
\item
Wang \cite{w} studied graphical translating solitons in $\real{m+1}$. He proved that for any
natural number $m$ there exist complete convex graphical translating solitons, defined in strip regions, which
are not rotationally symmetric. When the dimension  $m$ is greater than $2$, Wang proved that there are
entire convex graphical translating solitons. On the other hand, Wang proved that any entire convex graphical
translating soliton in $\real{3}$ must be rotationally symmetric in an appropriate coordinate system. It
is still an open problem, if {\it any entire  graphical translating soliton in $\real{3}$ (not necessarily convex)
is rotationally symmetric}.
\end{enumerate}

\subsection{The tangency principle}
A basic tool employed in the proof of one of the main theorems is the tangency principle.  According to this principle two
different translating solitons cannot \textquotedblleft touch" each other at one interior or boundary point. More precisely,
\begin{theorem}\label{tangency}
Let $\Sigma_1$ and $\Sigma_2$ be $m$-dimensional embedded translating solitons with boundaries $\partial\Sigma_1$ and $\partial\Sigma_2$
in the euclidean space $\real{m+1}$.
\begin{enumerate}[(a)]
\item $(${\bf Interior principle}$)$ Suppose that there exists a common point $x$ in the interior of $\Sigma_1$ and $\Sigma_2$ where the corresponding tangent spaces coincide.
Then $\Sigma_1$ coincides with $\Sigma_2$.
\medskip
\item $(${\bf Boundary principle}$)$
Suppose that the boundaries $\partial\Sigma_1$ and $\partial\Sigma_2$ lie in the same hyperspace $\Pi$ of $\real{m+1}$ and that there exists
a common point of $\partial\Sigma_1$ and $\partial\Sigma_2$ where $\Sigma_1$ and $\Sigma_2$ have the same tangent space. Then $\Sigma_1$
coincides with $\Sigma_2$.
\end{enumerate}
\end{theorem}
\begin{proof}
It is well known that translating solitons can be considered as minimal hypersurfaces in the conformally changed metric
$$\operatorname{G}_p:=e^{\frac{2}{m}\langle p,\v\rangle}\langle\cdot\,,\cdot\rangle$$
on $\real{m+1}$. Therefore, translators are real analytic hypersurfaces. Consequently, if two translators coincide in an open
neighborhood they should coincide everywhere. The theorem follows now from the interior and boundary tangency principle for embedded
minimal hypersurfaces in a Riemannian manifold. For a nice exposition of these principles we recommend the beautiful paper of Eschenburg
 \cite[Theorem 1 and Theorem 1a]{eschenburg}.
\end{proof}

\subsection{Global characterizations} We shall conclude this section with some characterizations of translators.

\begin{mythm}
Let $f:M^m\to\real{m+1}$ be a translating soliton which is not a minimal hypersurface.
Then $f(M^m)$ is a grim hyperplane if and only if the function $|A|^2H^{-2}$ attains a local maximum
on $M^m-\{H=0\}$.
\end{mythm}
\begin{proof}
Inspired by ideas developed by Huisken \cite{huisken}, consider the smooth function $h:U\to\real{}$, $U:=M-\{ H=0 \}$,
given by $h=|A|^2H^{-2}.$ Notice at first that
$$\nabla h=\frac{\nabla|A|^2}{H^2}-2h\frac{\nabla H}{H},$$
provided $H\neq 0$.
Let us now relate the Laplacian of the function $h$ with the quantities $H$, $A$ and $u$. We have,
\begin{eqnarray*}
\Delta h&=&H^{-2}\Delta|A|^2+|A|^2\Delta H^{-2}+2\langle\nabla|A|^2,\nabla H^{-2}\rangle\\
&=&H^{-2}\Delta|A|^2-2H^{-3}|A|^2\Delta H\\
&&+6H^{-4}|A|^2|\nabla H|^2-4H^{-3}\langle\nabla|A|^2,\nabla H\rangle.
\end{eqnarray*}
 By virtue of Lemma \ref{lemm rel}, we get
\begin{equation*}
\Delta h+H^{-1}\langle\nabla h,H\nabla u+2\nabla H\rangle-2H^{-4}Q^2=0,
\end{equation*}
where
\begin{eqnarray*}
Q^2:&=&H^2|\nabla A|^2+|A|^2|\nabla H|^2-H\langle\nabla H,\nabla|A|^2\rangle \nonumber\\
&=&\big|H\nabla A-\nabla H\otimes A\big|^2.
\end{eqnarray*}

Since $h$ attains a local maximum, by the strong maximum principle we obtain now that $h$ is constant
and $Q^2$ is identically zero. Thus,  for any triple of indices $i,j,k$, it holds
\begin{equation}\label{q}
H(\nabla_{e_i}A)(e_j,e_k)-e_i(H)A(e_j,e_k)=0,
\end{equation}
where here $\{e_1,\dots,e_m\}$ is a local orthonormal frame in the tangent bundle of the hypersurface.
From the last identity and the Codazzi equation we see that
\begin{equation}\label{q2}
e_{i}(H)A(e_j,e_k)-e_j(H)A(e_i,e_k)=0.
\end{equation}
for any triple of indices $i,j,k.$

{\bf Case 1:} Assume at first that $H$ is constant. Since by assumption $f$ is not minimal, this
constant is non-zero. Then, from equation (\ref{q}) it follows that
$$|\nabla A|=0.$$
Then, all the principal curvatures of $f$ are constant. Due to a theorem of Lawson \cite[Theorem 4]{lawson}, it follows that
$f(M^m)$ is locally isometric to a round sphere or to a product of a round sphere with a euclidean factor. However, none
of these examples are translators and consequently this situation cannot happen.

{\bf Case 2:} Assume now that there is a simply connected neighborhood $V$ where $|\nabla H|$ is not zero.
In the neighborhood $V$, choose the frame field $\{e_1,\dots,e_m\}$ such that
$$e_1=\frac{\nabla H}{|\nabla H|}.$$
Then, from equation (\ref{q2}), we obtain
$A(e_j,e_k)=0,$
for any $k\ge 1$ and $j\ge 2$. Therefore, $f$ has only one non-zero principal curvature. Denote now by
$\mathscr{D}:M^m\to TM^m$,
$$\mathscr{D}(x):=\{v\in T_xM^m:A(v,\cdot\,)=0\},$$
the nullity distribution and by
$$\mathscr{D}^{\perp}:=\operatorname{span}\{e_1\},$$
its orthogonal complement. It is well known that the distribution $\mathscr{D}$ is smooth
and integrable. Moreover, $\mathscr{D}$ is an autoparallel distribution, that is for any $X,Y\in\mathscr{D}$ it follows that
$\nabla_{X}Y\in\mathscr{D}$. The integral submanifolds of $\mathscr{D}$ are totally geodesic in $M^m$ and their images
via the immersion $f$ are totally geodesic submanifolds of $\real{m+1}$. Furthermore, the Gau{\ss} map of the immersion
$f$ is constant along the leaves of $\mathscr{D}$. A classical reference for the proofs of these facts is the paper of Ferus
\cite[Lemma 2, p. 311]{ferus}.

We claim now that $\mathscr{D}^{\perp}$ is also autoparallel and its integral curves are geodesics of $M$. Indeed,
fix a point $x_0$ and set $e=e_1(x_0)$. Extend the vector $e$ by parallel transport to a vector field along
the integral curve
$\gamma:(-\varepsilon,\varepsilon)\to M^m$ of $e_1$ passing through $x_0$. Let $B:TM^m\to TM^m$ be the Weingarten
operator associated to $A$. From the equation (\ref{q}), we obtain that
$$(\nabla_{e_1}B)v=e_1(H)v,$$
for any $v$ in $TM^m$.

A straightforward computation yields
\begin{eqnarray*}
\dt\left(|(B-H\cdot\operatorname{Id})e|^2\right)&=&2\langle(\nabla_{e_1}B)e-e_1(H)e,Be-He\rangle\\
&=&2\langle e_1(H)e-e_1(H)e,Be-He\rangle\\
&=&0.
\end{eqnarray*}
Therefore $e=e_1$ is a parallel vector field along the integral curves of $e_1$.
Thus,
$$\nabla_{e_1}e_1=0$$
and so the orthogonal complement $\mathscr{D}^{\perp}$ of $\mathscr{D}$ is autoparallel in $TM^m$ .

Observe now that
$$TM^m=\mathscr{D}\oplus\mathscr{D}^{\perp}.$$
We claim now that both the above distributions are parallel. Indeed, let $X,Y\in\mathscr{D}$ and $Z\in\mathscr{D}^{\perp}$.
Then,
$$0=X\langle Y,Z\rangle=\langle\nabla_{X}Y,Z\rangle+\langle Y,\nabla_XZ\rangle.$$
Hence, for any $X\in TM^m$ and $Z\in\mathscr{D}^{\perp}$ we have that $\nabla_XZ\in\mathscr{D}^{\perp}$.
This means that the distribution $\mathscr{D}^{\perp}$ is actually parallel. Similarly, one can show that the
nullity distribution $\mathscr{D}$ is also parallel. Hence, form the de Rham decomposition theorem, the hypersurface $f(M^m)$ splits as the Cartesian product
of a plane curve $\Gamma$ with a euclidean factor $\real{m-1}$.  Obviously,
this planar curve $\Gamma$ must be a grim reaper. Consequently, $f(M^m)$ contains an open neighborhood that is
part of a grim hyperplane. Because of the real analyticity, according to the analytic continuation theorem
\cite[Theorem1, p.213]{sampson} we deduce that $f(M^m)$ should coincide everywhere with a grim
hyperplane. This completes the proof.
\end{proof}

\begin{corollary}
Let $f:M^m\to\real{m+1}$ be a translating soliton with zero scalar curvature. Then either
$f(M^m)$ is a grim hyperplane or $f(M^m)$ is a totally geodesic hyperplane tangent to $\v$.
\end{corollary}
\begin{proof}
Note that under our assumptions,
$$|A|^2=H^2-\operatorname{scal}=H^2.$$
Hence, if $H$ is identically zero then $|A|^2$ is identically zero and $f(M^m)$ must be a totally geodesic hypersurface
tangent to $\v$.
Suppose now that there is a point where $H$ is not zero. In this case, there is an open neighborhood
where  the function $|A|^2H^{-2}$ is well-defined and equals $1$.
Consequently, the above theorem implies that $f(M^m)$ coincides with a grim
hyperplane.
\end{proof}
Similarly we can prove the following:
\begin{corollary}
Let $f:M^m\to\real{m+1}$ be a translating soliton with $H>0$ and $\operatorname{scal}\ge 0$. Then either
$f(M^m)$ is a grim hyperplane or $\operatorname{scal}>0$ everywhere.
\end{corollary}

\begin{theorem}
Let $f:M^m\to\real{m+1}$ be a weakly-convex translator. If there is a point
where the Gau{\ss}-Kronecker curvature vanishes, then the Gau{\ss}-Kronecker curvature vanishes everywhere.
\end{theorem}
\begin{proof}
We have that
$$\Delta A+\langle \nabla A,\nabla u\rangle+|A|^2A=0.$$
By the assumptions $A$ is a non-negative symmetric $2$-tensor. If there is a point where the smallest
principal curvature vanishes, from the strong elliptic maximum principle for tensors (see for example \cite{hamilton}
or  \cite[Section 2]{savas}), we get that the smallest principal curvature of $f$ vanishes everywhere.
\end{proof}

\section{Uniqueness of the translating paraboloid}\label{alexandrov}

The aim of this section is to show that a complete embedded translating soliton of the mean curvature
flow with a single end which is asymptotic to the rotationally symmetric translating paraboloid, must be
a translating paraboloid. The proof exploits the method of moving planes which was first introduced by
Alexandrov \cite{alexandrov} for the investigation of compact hypersurfaces with constant mean curvature
in a euclidean space. However, our approach follows ideas developed by Schoen \cite{schoen} where he
applied the method of moving planes to minimal hypersurfaces of the euclidean space.

\subsection{A uniqueness theorem}
Before stating and proving the main result of this section we have to introduce some notation and definitions.
Denote by $\pg:\real{m+1}\to\Pi$ the orthogonal projection
to the plane
$$\Pi:=\{(x_1,\dots,x_{m+1})\in\real{m+1}:x_1=0\},$$
that is $$\pg(x_1,\dots,x_{m+1}):=(0,x_2,\dots,x_{m+1}).$$
\begin{definition}\label{setbigger}
Let $A$ and $B$ be two arbitrary subsets of $\real{m+1}$. We say that the set $A$ is on the \textit{right hand side} of $B$
and write $A\ge B$ if and only if for every point $x\in\Pi$ for which
$$\pg^{-1}(x)\cap A\neq\emptyset\quad\text{ and }\quad\pg^{-1}(x)\cap B\neq\emptyset,$$
we have that
\begin{equation*}
\inf\big[x_1\big\{\pg^{-1}(x)\cap A\big\}\big]\ge\sup\big[x_1\big\{\pg^{-1}(x)\cap B\big\}\big],
\end{equation*}
where here $x_1\{P\}$ denotes the $x_1$-component of the point $P\in\real{m+1}$.
\end{definition}
Note that \textquotedblleft\,$\ge"$ is not a well ordered relation.
Observe also that there are subsets of $\real{m+1}$ that cannot be related to each other. In the matter of fact this relation is neither reflexive nor transitive.

Consider now the family of planes $\{\Pi(t)\}_{t\ge 0}$ given by
$$\Pi(t):=\{(x_1,\dots,x_{m+1})\in\real{m+1}:x_1=t\}.$$
Given a subset $A$ of $\real{m+1}$ let us also define the following subsets:
\begin{eqnarray*}
\delta_t(A):&=&\{(x_1,\dots,x_{m+1})\in A\; : \; x_1=t \}=A\cap\Pi(t),\\
A_+(t):&=&\{ (x_1,\dots,x_{m+1}) \in A \; : \; x_1 \geq t \},\\
A_-(t):&=&\{ (x_1,\dots,x_{m+1}) \in A \; : \; x_1 \leq t \},\\
A_+^\ast(t):&=&\{ (2 t-x_1,\dots,x_{m+1})\in\real{m+1}:  (x_1,\dots,x_{m+1}) \in A_+(t) \},\\
A_-^\ast(t):&=&\{ (2 t-x_1,\dots,x_{m+1})\in\real{m+1}:  (x_1,\dots,x_{m+1}) \in A_{-}(t) \},\\
Z_{t}:&=&\{ (x_1,\dots,x_{m+1}) \in\real{m+1}: x_{m+1} > t \}.
\end{eqnarray*}
Note that $A_{+}(t)$ are elements of $A$ that are on the right hand side of the plane $\Pi(t)$
and $A_{-}(t)$ are those elements of $A$ that belong to the left hand side of $\Pi(t)$. The
subset $A^{*}_{+}(t)$ is the reflection of $A_{+}(t)$ with respect to the plane $\Pi(t)$
while $A^{*}_{-}(t)$ stands for the reflection of $A_{-}(t)$ with respect to $\Pi(t)$.

\begin{mythm}\label{reflection}
Let $f:M^m\to\real{m+1}$ be a complete embedded translating soliton of the mean curvature flow
with finite genus and a single end that is smoothly asymptotic to a translating paraboloid. Then
$M=f(M^m)$ is a translating paraboloid.
\end{mythm}

\begin{proof}
For the sake of intuition we will present the proof for $m=2$.  The arguments for the general case are analog. From our assumptions it follows that there exists a positive real number $r$ such that $M-B(0,r)$ can
be written as the graph of a function
\begin{equation}\label{eq:g}
g(x_1,x_2)=\frac{1}{2}\big(x^2_1+x^2_2\big)-\frac{1}{2}\log\big(x^2_1+x^2_2\big)+O\left(\frac{1}{\sqrt{x^2_1+x^2_2}}\right),
\end{equation}
where here $B(0,r)$ stands for the open euclidean ball of $\real{3}$ which is centered at the origin and has radius $r$.
Consider now the vectors
$$v_{\vartheta}:=(\cos\vartheta,\sin\vartheta,0),$$
where here $\vartheta\in[0,2\pi)$. Our goal is to show
that for any angle $\vartheta$ the translator $M$ is symmetric with respect to the plane perpendicular to
$v_\vartheta$ passing through the origin of $\real{3}$. Since the rotations around the $x_3$-axis preserve the property of being a translating soliton, we deduce that it suffices to prove the
symmetry only along the plane $\Pi$. To this end, consider the set
$$\mathcal{A}:=\big\{t\in[0,\infty):M_+(t)\text{ is a graph over $\Pi$ and }M^*_{+}(t)\ge M_{-}(t) \big\}.$$
The proof will be finished if we can show that $0$ is contained in the set $\mathcal{A}$. This will be achieved by proving that $\mathcal{A}$ is a non-empty open and
closed subset of the interval $[0,+\infty)$. The proof of this fact will be concluded by several claims.
\begin{figure}[h]\includegraphics[scale=.08]{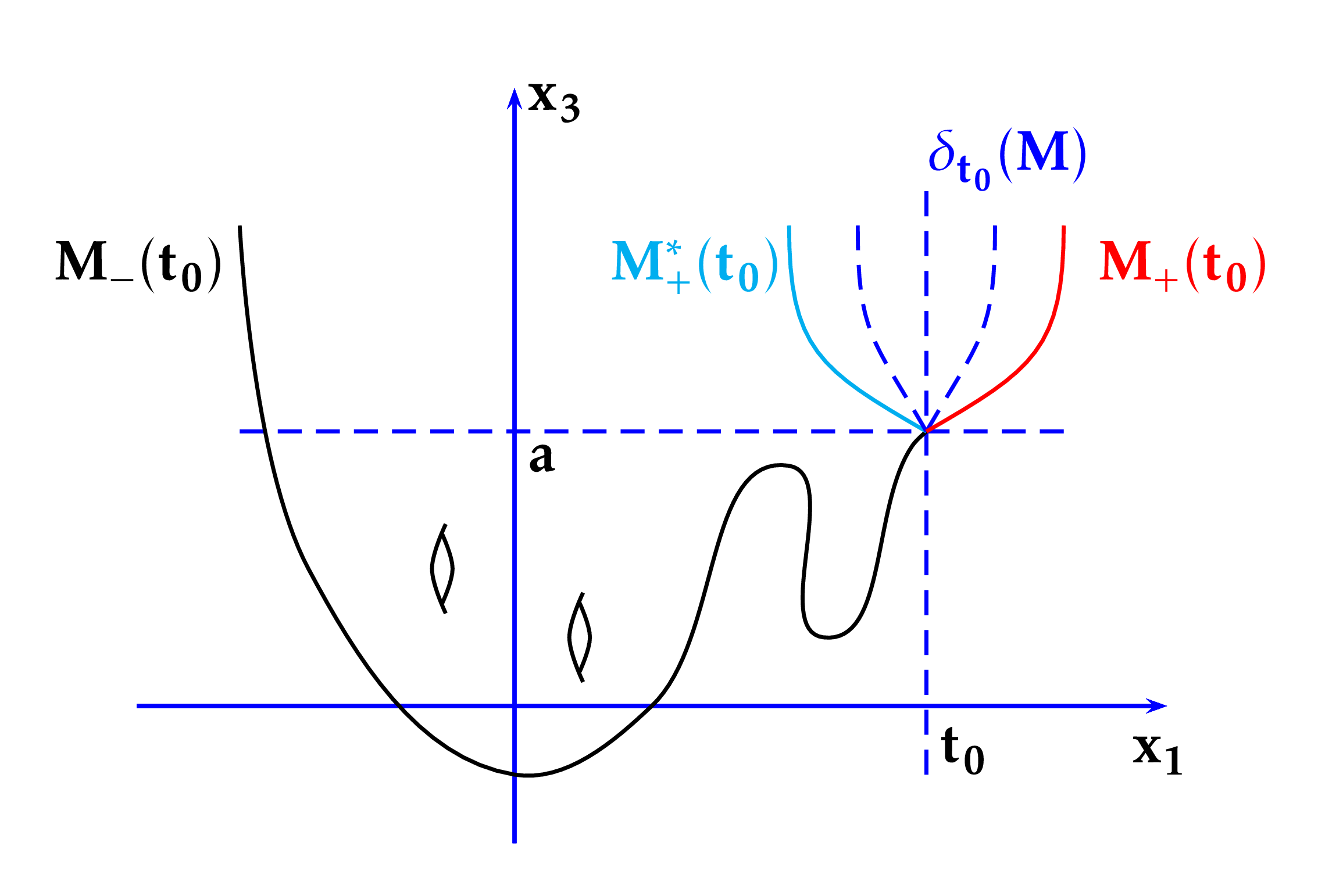}\label{reflection2}\end{figure}

{\bf Claim 1.} {\it The set $\mathcal{A}$ is not empty. In particular, there exists a positive number $t_1$ such that $[t_1,+\infty)\subset\mathcal{A}$.}

We will exploit the asymptotic behavior of the end of the translator to prove the existence of $t_1$. Indeed, choose the radius $r$ sufficiently large.
In fact, one can choose $r$ so large such that the part of the translator $M\cap Z_{r}$ is a horizontal graph (i.e. a graph over the $x_1x_2$-plane)
and that $M_+(r)\cap Z_r, M_-(r)\cap Z_{r}$ are both vertical graphs over $\Pi(r)$.

Choose a number $t_2>r$. Then, the set $M_{+}(t)$ sits outside the ball $B(0,r)$ for any $t\ge t_2$. Because the function
$$\varphi(s):=\frac{1}{2}s^2-\log(s)+O\left(\frac{1}{s}\right)$$
is strictly increasing for sufficiently large values of $s$, we deduce that for sufficiently large $r$
the set $M_{+}(t)$ can be represented as a graph over the plane $\Pi$
for any $t\ge t_2$.

Take $t_1:=2t_2$ and fix a $t\in[t_1,\infty)$. We claim that $M^*_{+}(t)\ge M_{-}(t)$. To prove
this let us represent $M^{*}_{+}(t)$ as the horizontal graph of the function $g_{t}$ given by the expression
\begin{eqnarray*}
g_t(x_1,x_2)&=&\frac{1}{2}\big\{(2t-x_1)^2+x^2_2\big\}\\
&&-\frac{1}{2}\log\big\{(2t-x_1)^2+x^2_2\big\}+O\left(\frac{1}{\sqrt{(2t-x_1)^2+x^2_2}}\right).
\end{eqnarray*}
Comparing the functions $g_t$ and $g$ we get that
\begin{eqnarray*}
g_t(x_1,x_2)&-&g(x_1,x_2)\\
&=&2t(t-x_1)-\frac{1}{2}\log\left\{\frac{4t(t-x_1)}{x^2_1+x^2_2}+1\right\}\\
&&+O\left(\frac{1}{\sqrt{(2t-x_1)^2+x^2_2}}\right)-O\left(\frac{1}{\sqrt{x^2_1+x^2_2}}\right)\\
&\ge&2t(t-x_1)-\frac{1}{2}\log\left\{\frac{4t(t-x_1)}{x^2_1+x^2_2}+1\right\}\\
&&-\frac{C}{\sqrt{(2t-x_1)^2+x^2_2}}-\frac{C}{\sqrt{x^2_1+x^2_2}},
\end{eqnarray*}
where $C$ is a positive constant. Recall that the above relation holds for
$$x^2_1+x^2_2>r^2.$$
In the case where $x_1<t$ and $t\ge t_1>r$ we have that
$$\log\left\{\frac{4t(t-x_1)}{x^2_1+x^2_2}+1\right\}\ge\frac{4t(t-x_1)}{x^2_1+x^2_2},$$
and
\begin{eqnarray*}
(2t-x_1)^2+x^2_2&=&(t+t-x_1)^2+x^2_2\\
&>&t^2+(t-x_1)^2+x^2_2\\
&>&r^2. 
\end{eqnarray*}
Therefore,
\begin{eqnarray*}
g_t(x_1,x_2)-g(x_1,x_2)&>&\frac{2t(t-x_1)(x^2_1+x^2_2-1)}{x^2_1+x^2_2}-\frac{2C}{r}\nonumber\\
&>&\frac{2r(t-x_1)(r^2-1)-2rC}{r^2}.
\end{eqnarray*}
Our aim is to compare $M^{*}_{+}(t)$ with $M_{-}(t)$ following the definition given in relation (\ref{setbigger}).
Observe now that there exists a positive constant $a$, not depending on $r$, such that $t-x_1>a$. Thus, we get that
\begin{equation}\label{asinto}
g_t(x_1,x_2)-g(x_1,x_2)>\frac{2ra(r^2-1)-2rC}{r^2}.
\end{equation}
Therefore, for large $r$, the right hand side of the above inequality is strictly positive. Hence,
\begin{eqnarray*}
&&\Big[M^{\ast}_{+}(t)\cap\big\{(x_1,x_2,x_3)\in\real{3}:x_1\le t_2\big\}\Big]\\
&&\quad\quad\quad\,\,\ge\Big[M_{-}(t)\cap\big\{(x_1,x_2,x_3)\in\real{3}:x_1\le t_2\big\}\Big].
\end{eqnarray*}
Moreover, from the fact that $M_{+}(t_2)$ is a graph over $\Pi$ we deduce that
\begin{eqnarray*}
&&\Big[M^{\ast}_{+}(t)\cap\big\{(x_1,x_2,x_3)\in\real{3}:t_2\le x_1\le t\big\}\Big]\\
&&\quad\quad\quad\,\,\ge\Big[M_{-}(t)\cap\big\{(x_1,x_2,x_3)\in\real{3}:t_2\le x_1\le t\big\}\Big].
\end{eqnarray*}
Hence, $[t_1,+\infty)\subset\mathcal{A}$ and this concludes the proof of the claim.

{\bf Claim 2.} {\it If $s \in \mathcal{A}$ then $[s, +\infty) \subset \mathcal{A}$.}

Suppose that $s\in\mathcal{A}$. This means that $M_{+}(s)$ can be represented as the graph of the height function
$u_s=x_1-s$
that is defined in the region $\Omega_s$. Notice that $\partial\Omega_s=\delta_{s}(M)$. The function $u_s$ cannot attain a local maximum
in the interior of $\Omega_s$, since otherwise we would have a contact point of $M$ with a plane parallel to $\Pi$. But then,
according to the interior tangency principle, $M$ is a plane and this contradicts our assumptions. Thus, the set $\Omega_s$ does not contain
compact connected components. In the matter of fact, because by assumption the end is asymptotic to a paraboloid we deduce that the set
$\Omega_s$ is a connected unbounded closed domain of the plane $\Pi(s)$. Let now $\tilde{s}>s$. Obviously the set $M_{+}(\tilde{s})$ is a
graph over $\Pi(s)$ since $M_{+}(\tilde{s})$ is a part of $M_{+}(s)$. Furthermore, because of the graphical condition of $M_{+}(s)$ we deduce that
$$\min \big[x_3\{\delta_s(M)\}\big]<\min\big[x_3\{\delta_{\tilde{s}}(M)\}\big],$$
where here $x_3\{P\}$ denotes the $x_3$-component of the point $P\in\real{3}$. Moreover, because $\Omega_{\tilde{s}}\subset\Omega_{s}$
and since $M^{\ast}_{+}(s)\ge M_{-}(s)$, we deduce that
$M^{\ast}_{+}(\tilde{s})\ge M_{-}(\tilde{s}).$
Consequently, $\tilde{s}\in\mathcal{A}$ for any $\tilde{s}\ge s$ and so $[s,\infty)\subset\mathcal{A}$.

{\bf Claim 3.} {\it The set $\mathcal{A}$ is a closed subset of the interval $[0,+\infty)$.}

Suppose that $\{t_n\}_{n\in\natural{}}$ is a sequence of points in $\mathcal{A}$ converging to  $t_0$. We have to
prove that $t_0$ is contained in $\mathcal{A}$. Indeed, at first notice that according to Claim 3 we have that
$(t_0,\infty)\subset\mathcal{A}$. Our goal is to prove that  $M_{+}(t_0)$ can be written as the graph of the
height function
$u_{t_0}=x_1-t_0$
which is defined in the region $\Omega_{t_0}$ of $\Pi(t_0)$ and that $M^{\ast}_{+}(t_0)\ge M_{-}(t_0)$.
Let us suppose to the contrary that there are points $P=(x_1,x_2,x_3)$ and $Q=(\tilde{x}_1,x_2,x_3)$ in $M_{+}(t_0)$ such that
$\tilde{x}_1>x_1$. Then we must have $x_1=t_0$ since $s\in\mathcal{A}$ for any $s>t_0$. We distinguish two cases now:
either $P$ belongs in the interior of $\Omega_{t_0}$ or $P$ belongs to the boundary $\delta_{t_0}(M)$ (see Figure 5 below).

\begin{figure}[h]\includegraphics[scale=.08]{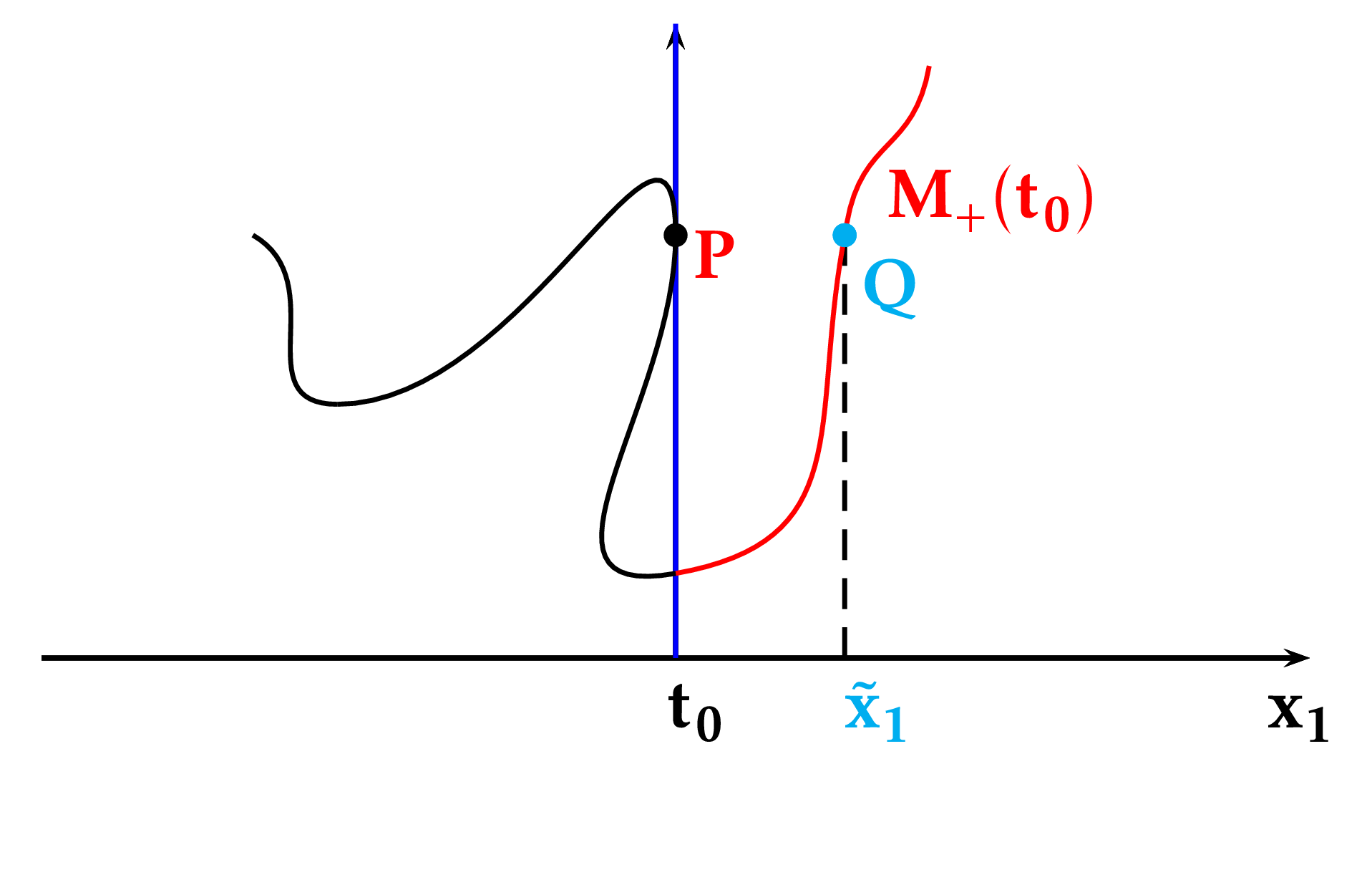}\includegraphics[scale=.08]{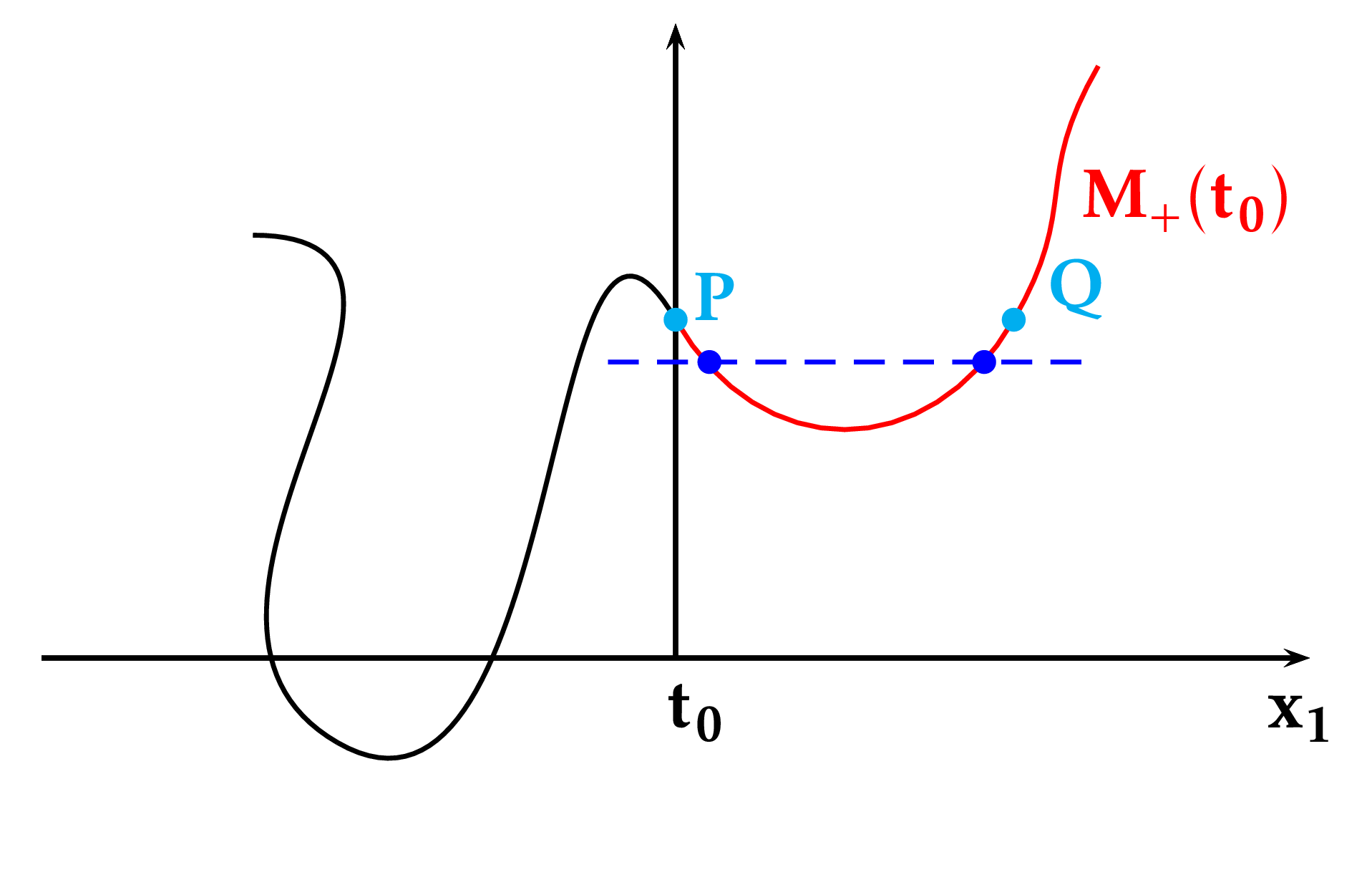}\caption{}\label{reflection1}\end{figure}
We shall show now that neither of the those alternatives can happen. Indeed:

{\it First Case:} Let $\varepsilon$ be a positive number 
such that
$$\tilde{x}_1>x_1+2\varepsilon=t_0+2\varepsilon.$$
From the last inequality we deduce that
$$2(t_0+\varepsilon)-\tilde{x}_1<t_0=x_1,$$
which contradicts the fact $M^{*}_{+}(t_0+\varepsilon)\ge M_{-}(t_0+\varepsilon)$.

{\it Second Case:} Consider the cylindrical solid
$$S:=\big\{(t;z):t\ge t_0,\, z\in \Omega_{t_0}\big\}.$$
Then $S$ intersects $M_{+}(t_0)$ at points outside the plane $\Pi(t_0)$. But this fact contradicts again the graphical property of the family of surfaces
$\{M_{+}(t_0+\varepsilon)\}_{\varepsilon>0}$.

Consequently, the set $M_{+}(t_0)$ can be represented as a graph of the height function over the plane $\Pi$. Moreover, because of
the continuity
$$M^{\ast}_{+}(t_0)\ge M_{-}(t_0).$$
Hence, $t_0\in\mathcal{A}$ and this completes the proof of the claim.

{\bf Claim 4.} {\it The minimum of the set $\mathcal{A}$ is $0$. In particular, $\mathcal{A}=[0,\infty)$.}

We argue again in this proof by contradiction. Suppose to the contrary that  $t_0:=\min\mathcal{A}>0$.
Then we will show that there exists a positive number $\varepsilon$ such that $t_0-\varepsilon\in\mathcal{A}$,
which will be the contradiction.

We will show at first that there exists a positive constant $\varepsilon_1<t_0$ such that $M_{+}(t_0-\varepsilon_1)$ is a graph.
Indeed, from the asymptotic behavior of the end we deduce that there exists positive number $a$ which is
sufficiently bigger than $r$ and such that
\begin{equation} \label{eq:N}
{\rm dist}_{\R^3} \Big[\xi \big\{M_+(t_0) \cap Z_{a} \big\},  \Pi \Big] >0,
\end{equation}
where $\xi:M \rightarrow \mathbb{S}^2$ stands for the Gauss map of $M$.
Note that there is $\varepsilon_0>0$ such that
$M_+(t_0-\varepsilon_0)\cap Z_a$
can be represented as a graph over the plane $\Pi$ and furthermore
\begin{equation} \label{eq:infinity} M_+^*(t_0-\varepsilon_0) 
\cap Z_a \geq M_-(t_0-\varepsilon_0) 
\cap Z_a.
\end{equation}
Consider now the compact set
$$\mathcal{K}:=M \cap \{(x_1,x_2,x_3)\in\real{3}:x_3\le a\}.$$
Taking into account that $t_0 \in \mathcal{A}$, we 
deduce that $\mathcal{K}_+(t_0)$ is a graph over the plane $\Pi$. 
At first notice that there is no point in $\mathcal{K}_+(t_0)-\delta_{t_0}(M)$ with normal vector included in the plane $\Pi$.
Indeed, if this was true then from the tangency principle at the boundary we would get that $M$ is symmetric around an
axis that is not passing through the origin which contradict the assumption on the end of $M$.
Moreover, the same argument yields that there is no point in $\mathcal{K}_+(t_0)\cap\delta_{t_0}(M)$ whose normal vector is lying in the plane $\Pi$. Consequently, 
$$\xi\big\{\mathcal{K}_+ (t_0)\big\} \cap \Pi=\emptyset.$$
Because the set $\mathcal{K}_+(t_0)$ is compact, 
there exists $\varepsilon_1 \in (0,\varepsilon_0]$ small enough such that, for all $t \in [t_0-\varepsilon_1,t_0],$
$$\xi\big\{\mathcal{K}_+(t)\big\} \cap \Pi=\emptyset,$$
Because of this fact and because of the compactness we deduce that the set $\mathcal{K}_+(t)$ can be represented
as graph $\Pi$ for every $t \in [t_0- \varepsilon_1,t_0]$. Consequently, $M_+(t)$ is a graph over the plane $\Pi$,
for all $t \geq t_0-\varepsilon_1$. Hence, the first step of the plan is finished.

Now we will conclude the plan by proving that there exist a positive constant $\varepsilon_2<\varepsilon_1$
such that $M^{\ast}_{+}(t_0-\varepsilon_2)\ge M_{-}(t_0-\varepsilon_2)$. Indeed, at first notice that
$$M_+^*(t)\cap M_-(t)\cap \mathcal{K} \subset \mathcal{K}_-(t_0-\varepsilon_0),$$
for all $t \geq t_0- \varepsilon_1.$
Secondly, because $M^{\ast}_{+}(t_0)\ge M_{-}(t_0)$, we have that
$$M_+^*(t_0) \cap M_-(t_0)=\delta_{t_0}(M).$$
Indeed this holds true, since otherwise any point lying in the set
$$\big\{M_+^*(t_0) \cap M_-(t_0)\big\}-\delta_{t_0}(M)$$
must be an interior point of contact  between $M_+^\ast(t_0)$ and $ M_-(t_0)$.
But then, from the interior tangency principle, we would have that
$$M_+^\ast(t_0) = M_-(t_0).$$
This is  leads to a contradiction because $t_0>0$ and  from (\ref{asinto}) we get that
$M_+^\ast(t_0)$ and $ M_-(t_0)$ are not asymptotic at infinity. Using now the fact that
$$M_+^*(t_0) \cap M_-(t_0)=\delta_{t_0}(M)$$
and that  $\mathcal{K}$ is compact, it follows that there exists a positive number 
$\varepsilon_2 \in (0,\varepsilon_1]$ such that 
$$M_+^*(t)\cap M_-(t)\cap \mathcal{K}= \delta_t(M) \cap \mathcal{K},$$
for all $t \geq t_0-\varepsilon_2$. Thus,
$$\big\{M_+^*(t)\cap M_-(t)\big\}-\delta_t(M) \subset M \cap Z_a,$$
and from  \eqref{eq:infinity} we get that
$$M_+^*(t)\cap M_-(t)=\delta_t(M),$$
for all $t \geq t_0-\varepsilon_2$. Then a continuity argument implies that
$$M_+^\ast(t) \geq M_-(t),$$
for all $t \geq t_0-\varepsilon_2$ and the second and final step of the plan is completed.

Hence, there exists a positive constant $\varepsilon<t_0$ such that $t_0-\varepsilon \in \mathcal{A}$.
This contradicts the assumption that $t_0$ is a minimum. Consequently, $t_0=0$
and the proof is finished.

From Claim 4 we get that $M_+^\ast (0) \geq M_-(0).$ A symmetric argument yields that
$M_-^\ast (0) \geq M_+(0)$. Therefore
$$M^{\ast}_{+}(0)=M_{-}(0)$$
and so $M$ is symmetric with respect to the plane $\Pi$. This completes the proof of the theorem.
\end{proof}

\begin{remark}
Using similar ideas to those applied in Theorem \ref{reflection} we can deduce some results about 
the asymptotic behavior of entire graphical translating solitons in $\real{m+1}$.
\begin{enumerate}[(a)]
\item
Let $M$ an entire graphical translating soliton satisfying the growth condition
$$|u(x)| \leq C \, |x|^\alpha, $$
for all $x \in \R^m-B(0,r)$.
Then $\alpha \geq 2.$
In order to prove this fact, we proceed again by contradiction. Suppose to the contrary
that there exists an entire graphical translator in the euclidean space $\real{m+1}$ satisfying the above growth
condition with $\alpha<2$. 
Let $X$ be the translating paraboloid of Example 2.2 (d)
and translate it vertically until the surfaces $X$ and $M$ do not intersect. Note that this is possible because
we are assuming $\alpha<2$ and $X$ has the asymptotic behavior described in \eqref{eq:g}.
Then, the paraboloid $X$ will move vertically downwards until there is a first point of contact
with the surface $M$. This first contact can not occur at infinity, because we are assuming $\alpha <2$.
Then, the tangency principle implies that $M$ should coincide with a translated copy of $X$. But
this is absurd, because $X$ is asymptotic to the graph over
\begin{equation*}
g(x)=\frac{1}{2}|x|^2-\log|x|+O\big(|x|^{-1}\big)
\end{equation*}
at infinity.
\medskip
\item
Exchanging the role of $M$ and $X$ in the above argument, 
then we can prove that: Suppose that $M$ is an entire graphical translator satisfying the following growth condition
$$|u(x)| \geq C \, |x|^\alpha,$$
for all $x\in \real{m}-B(0,r)$. Then, $\alpha \leq 2.$
\medskip
\item
Using again the tangency maximum principle we can show that there are no complete
and embedded translators that are contained in the solid half-cylinder
$$\quad\quad\mathscr{C}:= \{(x_1, \ldots,x_{m+1}) \in \mathbb{R}^{m+1} \; : \; x_1^2+ \cdots +x_m^2 \leq r^2, \; x_{m+1}>0 \}.$$
\item
The reason that the mean curvature flow of compact hypersurfaces in $\real{m+1}$ form singularities is the following comparison
principle: {\it Let $M_1$ and $M_2$ be two compact Riemannian manifolds of dimension $m$ and $f:M_1\to\real{m+1}$, $g:M_2\to\real{m+1}$
disjoint isometric immersions. Then also the solutions $f_t$ and $g_t$ of the mean curvature flow
remains disjoint}. The compactness assumption cannot be relaxed with that of completeness. Indeed, take as $f:M_1\to\real{3}$
be the unit euclidean sphere and as $g:M_2\to\real{3}$ a complete minimal surface lying inside the unit ball. Such examples were first
constructed by Nadirashvili \cite{nadirashvili}. Obviously $f$ and $g$ do not have intersection points. However, under the mean curvature flow, $f$
shrinks to a point in finite time while $g$ remains stationary.
\end{enumerate}
\end{remark}

\section{The two-dimensional case}
In this section we will investigate $2$-dimensional translating solitons lying in the two dimensional case. We will study the
Gau{\ss} image of such surfaces and will obtain several classification results as well as topological obstructions for the
existence of translating solitons of the mean curvature flow in the euclidean space $\real{3}$.

\subsection{Gau{\ss} image of punctured Riemann surfaces}
It is a well known fact that any smooth oriented compact surface is diffeomorphic
either to a sphere or to a torus with $g$ holes. The Euler characteristic of
such compact surface $\Sigma_g$ depends only on the genus $g$ and is given by the formula
$$\mathcal{X}(\Sigma_g)=2-2g.$$
This classification can be extended also to compact surfaces $\Sigma$ with
boundary. Note that from compactness,  the boundary $\partial \Sigma$ has a finite number $k$ of components.
Moreover, each boundary component of $\Sigma$ is a connected compact $1$-manifold, that is a circle. Now, if we take
$k$ closed discs and glue the boundary of the $i$-th disc to the $i$-th component of the boundary of $\Sigma$, we
obtain a smooth compact surface which we will denote with the letter $\Sigma^*$. It turns out that the topological type of
a compact surface with boundary $\Sigma$ depends only on the number $k$ of its boundary components and the topological
type of the surface $\Sigma^*$. More precisely, $\Sigma$ is diffeomorphic to $\Sigma_{g,k}$,
where
$$\Sigma_{g,k}=\Sigma_g\smi\{p_1,\dots,p_k\}$$
is a punctured Riemann surface given by a closed Riemann surface $\Sigma_g$ of genus $g$ with $k$ points $p_1,\dots,p_k\in\Sigma_g$ removed.

The Euler characteristic of a compact surface with boundary is defined exactly in the same way as in the case of a compact
surface without boundary. It follows that,
$$\mathcal{X}(\Sigma_{g,k})=\mathcal{X}(\Sigma)=\mathcal{X}(\Sigma_g)-k=2-2g-k.$$
The {\it genus  of a compact surface $\Sigma$ with boundary}  is defined to be the genus of the compact surface
$\Sigma_g=\Sigma^*$.

From now we will focus on Riemann surfaces of the form
$$\Sigma_{g,k}:=\Sigma_g\smi\{p_1,\dots,p_k\},$$
where $\Sigma_g$ is a compact Riemann surface of genus $g$ and $p_1,\dots,p_k\in\Sigma_g$ are $k$ distinct punctures. Recall that an {\it end} of the
surface $\Sigma_{g,k}$ is a diffeomorphism $\tilde\varphi_j:D\to\Sigma_g$ of the closed unit disc $D$  to $\Sigma_g$ such that $\tilde\varphi_j(0)=p_j$
for one index $j\in\{1,\dots,k\}$ and
$\tilde\varphi_j(D)\cap\{p_1,\dots,p_k\}=\{p_j\}.$
We will denote by
$\varphi_j:D^{^*}\to\Sigma_{g,k}$ the maps $\varphi_j:=\bigl(\tilde\varphi_j\bigr)_{|D^{^*}},$
where
$D^{^*}:=\{x\in\real{2}:0<|x|\le 1\}$
denotes the punctured unit disc. The surface
$\Sigma_{g,k}$ will be called a {\it planar domain}, if $g=0$ and $k\ge 1$, that is if $\Sigma_{g,k}$ is a punctured sphere.

In the following lemma we present a result which will allow in the rest of the paper a surgery argument for cylindrical ends.

\begin{lemma}[Spherical cap lemma]\label{spherical}
Let $\alpha_{t}:\mathbb{S}^1\to\real{2}$, $t\in[0,1]$ be a smooth family of closed embedded curves and let $\operatorname{d}_0$ denote the diameter of $\alpha_0$.
Define the cylinder $\Theta:\mathbb{S}^1\times[0,1]\to\real{3}$ given by
$$\Theta(s,t):=(\alpha_t(s),t).$$
Then for any $\sigma>0$ there exists an $\varepsilon\in(0,\sigma)$ and
a smooth embedding $C:D\to\real{3}$ of the closed unit disc $D$ such that
\begin{enumerate}[\rm(a)]
\item for all $x\in D$ with $|x|\in[1-\varepsilon,1]$, it holds
$$C(x)=\Theta(x/|x|,1-|x|).$$
\item The height function $u$ given by
$u=\langle C,\operatorname{e}_3\rangle$
satisfies $0\le u\le\sigma$.
Moreover $u(0)=\sigma$ and the critical set of $u$ is $\operatorname{Crit}(u)=\{0\}$.
\smallskip
\item The diameter $\operatorname{d}_C$ of the embedded disc satisfies
$$\operatorname{d}_C\le 2\sigma+\operatorname{d}_0.$$
\item The Gau\ss\ curvature $K$ at $C(0)$ is positive.
\smallskip
\end{enumerate}
\begin{figure}[h]
\includegraphics[scale=.05]{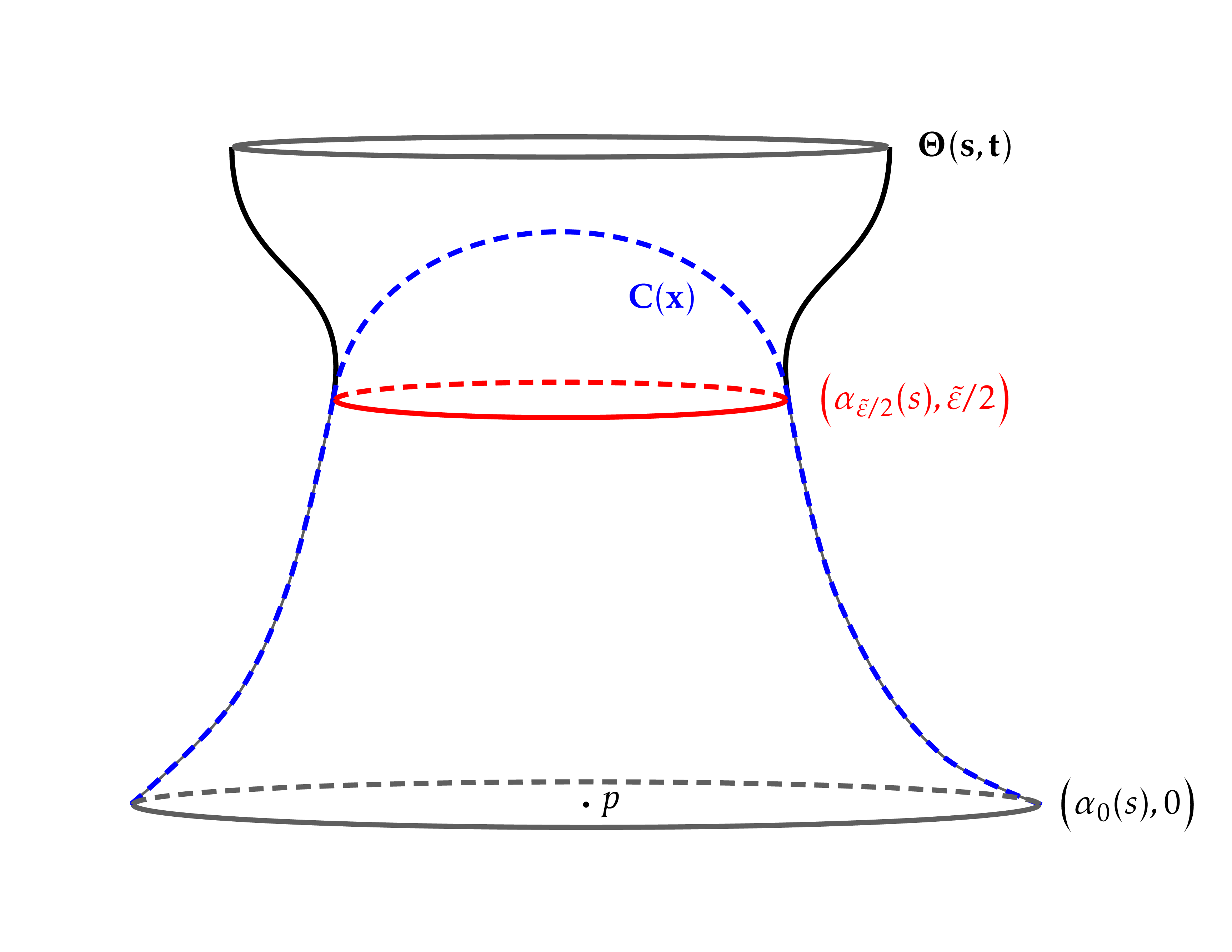}\caption{Adding a spherical cap to a cylinder}\label{gluing2}
\end{figure}
\end{lemma}
\begin{proof}
At first let us state some basic facts that will be used in the proof of the lemma.

{\bf Fact 1.} 
Let $p\in\real{2}$ be an arbitrary point in the interior of the curve $\alpha_0$ such that
$$|\alpha_0(s)-p|\le\frac{\operatorname{d}_0}{2}$$
for all $s\in\mathbb{S}^1$.
For $\sigma>0$ we can find an $\tilde\varepsilon\in(0,\min\{1,\sigma\})$ such that for all $t\in[0,\tilde\varepsilon]$
it holds
$$|\alpha_t(s)-p|\le\sigma+\frac{\operatorname{d}_0}{2}.$$
{\bf Fact 2.} According to a well-known theorem of Grayson \cite{grayson}, the
curve shortening flow provides a smooth isotopy of a closed embedded curve $\gamma_0$ to a single point $q\in\real{2}$
in the interior of $\gamma_0$ by smooth embedded curves. Moreover, all evolved
curves satisfy the inequality
$$\max_{\mathbb{S}^1}|\gamma_\tau-q|\le\max_{\mathbb{S}^1}|\gamma_0-q|.$$
Furthermore, the above inequality is still valid, if one blows up the solutions homothetically around $q$ by a time
dependent factor so that the length of the evolving curve is fixed. Under this rescaling the curves become
circular in the limit.

Starting with the curve $\gamma_0:=\alpha_{\tilde\varepsilon/2}$, from Fact 1 and Fact 2 we see that there exists a smooth
isotopy  $\beta_{t}:\mathbb{S}^1\to\real{}$, $t\in[0,1]$, 
such that:
\begin{itemize}
\item
$\beta_t=\alpha_t$ for $t\in[0,\tilde\varepsilon/2]$,
\smallskip
\item $\beta_t(s)=p+(1-t)e^{is}$, for $(s,t)\in\mathbb{S}^1\times[1-\tilde\varepsilon/2,1],$
\smallskip
\item $\beta_t:\mathbb{S}^1\to\real{2}$ is a closed embedded curve for all $t\in[0,1)$ with
$$\max_{\mathbb{S}^1}|\beta_t-p|\le \sigma+\operatorname{d}_0/2$$
for all $t\in[0,1]$.
\end{itemize}
Choose a smooth function $\phi:[0,1]\to\real{}$ with the following properties
\begin{figure}[h]
\includegraphics[scale=.08]{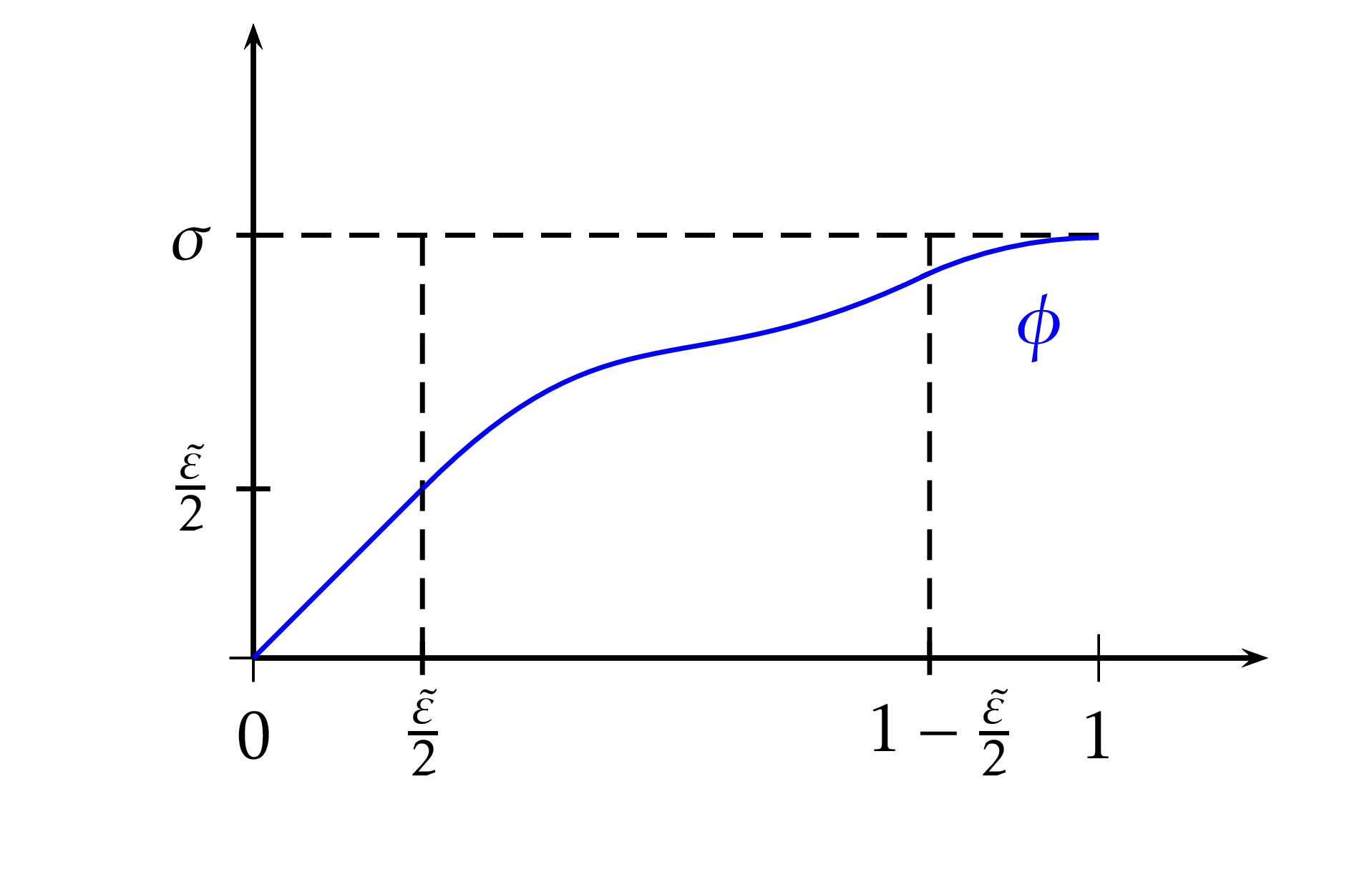}\caption{Smoothing function}\label{smoothing}
\end{figure}
\begin{itemize}
\item $\phi(t)=t$ for $t\in[0,\tilde\varepsilon/2]$,
\smallskip
\item  $\phi(t)=\sqrt{\sigma^2-(1-t)^2}$ for $t\in[1-\tilde\varepsilon/2,1]$,
\smallskip
\item  $\phi'(t)>0$ for $t\in[0,1)$.
\end{itemize}
Let us now define the map $C:D\to\real{2}$ by
$$C(x)=\bigl(\beta_{1-|x|}(x/|x|),\phi(1-|x|)\bigr).$$
If we set $\varepsilon:=\tilde\varepsilon/2$, then for any $x\in D$ with $|x|\in[1-\varepsilon,1]$ we get
\begin{eqnarray*}
C(x)&=&\bigl(\beta_{1-|x|}(x/|x|),\phi(1-|x|)\bigr)\\
&=&\bigl(\alpha_{1-|x|}(x/|x|),1-|x|\bigr)\\
&=&\Theta\bigl(x/|x|,1-|x|\bigr).
\end{eqnarray*}
The last equality proves assertion {\rm (a)} of the Lemma. Since the height function is given by
$$u(x)=\langle C(x),\operatorname{e}_3\rangle=\phi(1-|x|),$$
from the properties of $\phi$ we immediately get {\rm (b)}. Finally, from the construction of the curves $\beta_t$,
we obtain that $C$ is contained in the ball of radius $\sigma+\operatorname{d}_0/2$ with center at $p$. Thus the
diameter $\operatorname{d}_C$ of the cap is bounded by $2\sigma+\operatorname{d}_0$, which
implies assertion {\rm (c)} of the lemma. That the Gau\ss\ curvature at the top is strictly positive, follows from
the fact that $C$ coincides with the portion of a round sphere close to the top. This proves assertion {\rm(d)} and
completes the proof of the lemma.
\end{proof}

In the following theorem we present a general theorem concerning the Gau{\ss} image of a complete surface of the
euclidean space $\real{3}$.

\begin{theorem}\label{gauss map}
Let $f:\Sigma_{g,k}=\Sigma_g\smi\{p_1,\dots,p_k\}\to\real{3}$ be a complete immersion
of a punctured Riemann surface. Suppose $\v\in\mathbb{S}^2$ is a fixed unit vector and let
$u:=\langle f,\v\rangle$ denote the height function of the surface
with respect to the direction $\v$. Suppose that for each puncture $p_j$, $1\le j\le k$, the following two conditions holds:
\begin{enumerate}[\rm(a)]
\item there exists a constant $\varepsilon>0$ such that
$\liminf_{x\to p_j}|\nabla u(x)|\ge\varepsilon.$
\medskip
\item the interval $\bigl(\liminf_{x\to p_j}u(x),\limsup_{x\to p_j}u(x)\bigr)$
does not coincide with the real line.
\end{enumerate}
Then either the image of the Gau\ss\ map $\xi:\Sigma_{g,k}\to\mathbb{S}^2$ contains the pair
$\{\v,-\v\}$ or the genus satisfies $g\le1$. If in addition $f$ is an embedding, then either the image of $\xi$ contains $\{\v,-\v\}$
or $g=0$ and thus $\Sigma_{g,k}$ is a planar domain. Moreover, in all cases we have
$$\limsup_{x\to p_j}|u(x)|=\infty$$
for all punctures $p_j$.
\end{theorem}
\begin{proof}
The idea of the proof is to glue spherical caps along each end and to apply degree theory to obtain informations about the image
of the Gau{\ss} map. We divide the proof into several steps:

{\bf Step 1.} To do the surgery, at first we need some information about the nature and the behavior of $u$ along the ends.

{\bf Claim:} {\it For each puncture $p\in\{p_1,\dots,p_k\}$
of $\Sigma_{g,k}$ there exists $\varepsilon>0$ and
a closed embedded curve $\alpha\subset\Sigma_{g,k}$ such that:
\begin{enumerate}[\rm(i)]
\item
The coordinate function $u$ is constant along the curve $\alpha$,
\medskip
\item
The closure of one of the connected components $V$ of $\Sigma_{g,k}\smi\, \alpha$ within $\Sigma_g$
is diffeomorphic to the closed unit disc $D\subset\real{2}$ such that
$\overline V\cap\{p_1,\dots,p_k\}=\{p\}$
and $\inf_V|\nabla u|\ge\varepsilon$.
\end{enumerate}}
Indeed,  from the assumption {\rm (a)} of the theorem it is clear that {\rm(ii)} can be fulfilled. That is,
there exists an  open set $U\in\Sigma_{g,k}$ such that:
\begin{itemize}
\item  $\overline U$ is diffeomorphic to $D$,
\smallskip
\item $\overline U\cap\{p_1,\dots,p_k\}=\{p\}$,
\smallskip
\item $\inf_{U}|\nabla u|\ge\varepsilon$ for a positive constant $\varepsilon>0$.
\end{itemize}
Hence it suffices to prove that $U$ contains a closed level set curve 
of $u$ which encloses $p$ in its interior, i.e. a curve like $\alpha_1$ in Figure \ref{levelsets}.
Let us denote by $\gamma_0$ the boundary of $U$, that is $\gamma_0:=\partial U$, and set
$$u_0^+:=\max_{x\in\gamma_0}u(x),\quad u_0^-:=\min_{x\in\gamma_0}u(x),\quad \delta_0:=u_0^+-u_0^-.$$
Let $x$ be a point in $U$ and
$$\operatorname{d}_x:=\operatorname{dist}(x,\gamma_0)$$
its distance from the boundary curve $\gamma_0$.
Denote by $\beta_x$ the flow line of $\nabla u/|\nabla u|$ with $\beta_x(0)=x$. Since the
vector field $\nabla u/|\nabla u|$ is well defined in $U$ and because $\operatorname{dist}(x,\gamma_0)=\operatorname{d}_x$, the
flow line can at least be parametrized over the interval $[0,\operatorname{d}_x]$ until it reaches the boundary curve $\gamma_0$ (if it reaches $\gamma_0$ at all).
We compute
\begin{eqnarray*}
\left|u(\beta_x(\operatorname{d}_x))-u(x)\right|
&=&\left|\int_0^{\operatorname{d}_x}\dt(u(\beta_x(t))dt\right|\\
&=&\left|\int_0^{\operatorname{d}_x}\langle\nabla u\circ\beta_x(t),\dt\beta_x(t)\rangle dt\right|\\
&=&\left|\int_0^{\operatorname{d}_x}|\nabla u\circ\beta_x(t)|dt\right|\\
&=&\int_0^{\operatorname{d}_x}|\nabla u\circ\beta_x(t)|dt\\
&\ge&\varepsilon\operatorname{d}_x.
\end{eqnarray*}
Since the surface $\Sigma_{g,k}$ is complete, the Hopf-Rinow Theorem implies that
there exists a point $q\in U$ such that $\operatorname{d}_x$ is arbitrarily large. So,
the last inequality shows
$$\limsup_{x\to p}|u(x)|=\infty.$$
In particular, if we choose $x\in U$
such that $\operatorname{d}_x>\delta_0/\varepsilon$,
then we see that there must exist another point $q\in U$ with
$u(q)\not\in[u_0^-,u_0^+]$. Let $\alpha$ be the connected component of the level set $u^{-1}(u(q))\cap (U\cup\gamma_0)$.
Since $\nabla u\neq 0$ on $U\cup\gamma_0$, $\alpha$ must be an embedded regular curve. Thus, one of the following cases
holds (cf. Figure \ref{levelsets}).
\begin{figure}[h]
\includegraphics[scale=.1]{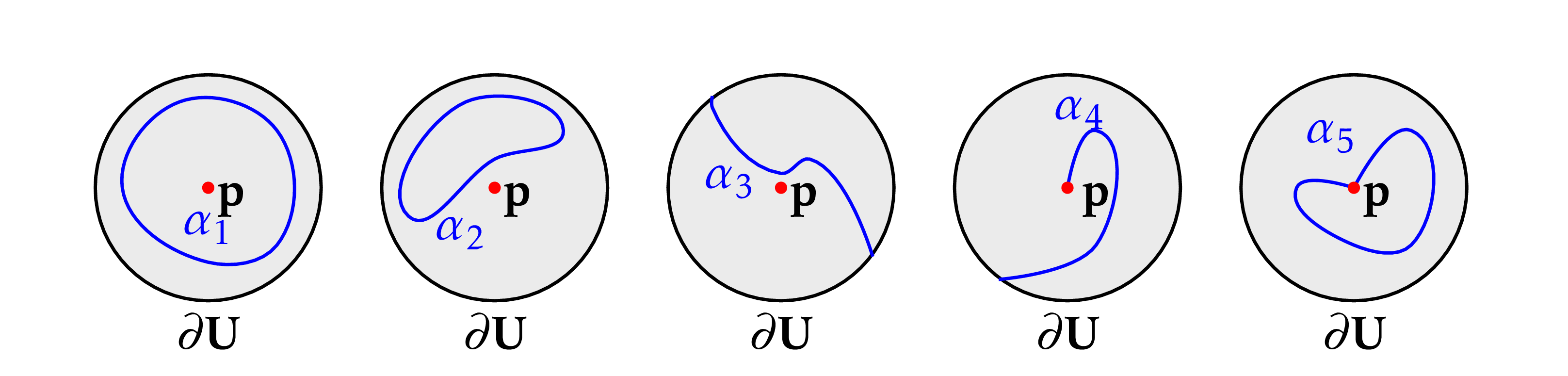}\caption{Possible types of level sets.}\label{levelsets}
\end{figure}
\begin{enumerate}
\item[(C1)] $\alpha\subset U$ and $p$ lies in the interior of $\alpha$.
\item[(C2)] $\alpha\subset U$ and $p$ lies in the exterior of $\alpha$.
\item[(C3)] $\alpha\cap\gamma_0\neq\emptyset$.
\item[(C4)] Both ends of $\alpha$ connect to $p$.
\end{enumerate}
The case (C2) cannot occur, because then the function $u$ would admit a local extremum in the interior of
$\alpha$ which in particular implies $\nabla u=0$ there. Case {\rm(C3)} is impossible since
$$u_{|\alpha}\equiv u(q)\not\in[u_0^-,u_0^+].$$
So we only need to exclude the last case {\rm(C4)}.
Choose a closed curve $\gamma_1\subset U$ such that $q\in\gamma_1$ and $\alpha\smi\{q\}\in\operatorname{int}(\gamma_1)$ (cf. Figure \ref{puncture}). Similar as above define
$$u_1^+:=\max_{x\in\gamma_1}u(x),\quad u_1^-:=\min_{x\in\gamma_1}u(x),\quad \delta_1:=u_1^+-u_1^-.$$ Applying once more the Hopf-Rinow Theorem, we can find a point
$\tilde q\in\alpha$ with $\operatorname{dist}(\tilde q,\gamma_1)>\delta_1/\epsilon$. Computing
as above we obtain for all points $q'$ on the flow line $\beta_{\tilde q}$ of $\nabla u/|\nabla u|$
the estimate
$$|u(q')-u(q)|=|u(q')-u(\tilde q)|\ge\epsilon\operatorname{dist}(\tilde q,\gamma_1)>\delta_1.$$
From this estimate we deduce that the flow line $\beta_{\tilde q}$ can never intersect the curve $\gamma_1$. Thus
$\beta_{\tilde q}\subset\operatorname{int}(\gamma_1)$ and since $|\nabla u|\ge \varepsilon$ 
this implies 
$$\liminf_{\beta_{\tilde q}}u=-\infty\quad\text{and}\quad\limsup_{\beta_{\tilde q}}u=+\infty$$
which by assumption {\rm (b)} is impossible.
\begin{figure}[h]
\includegraphics[scale=.06]{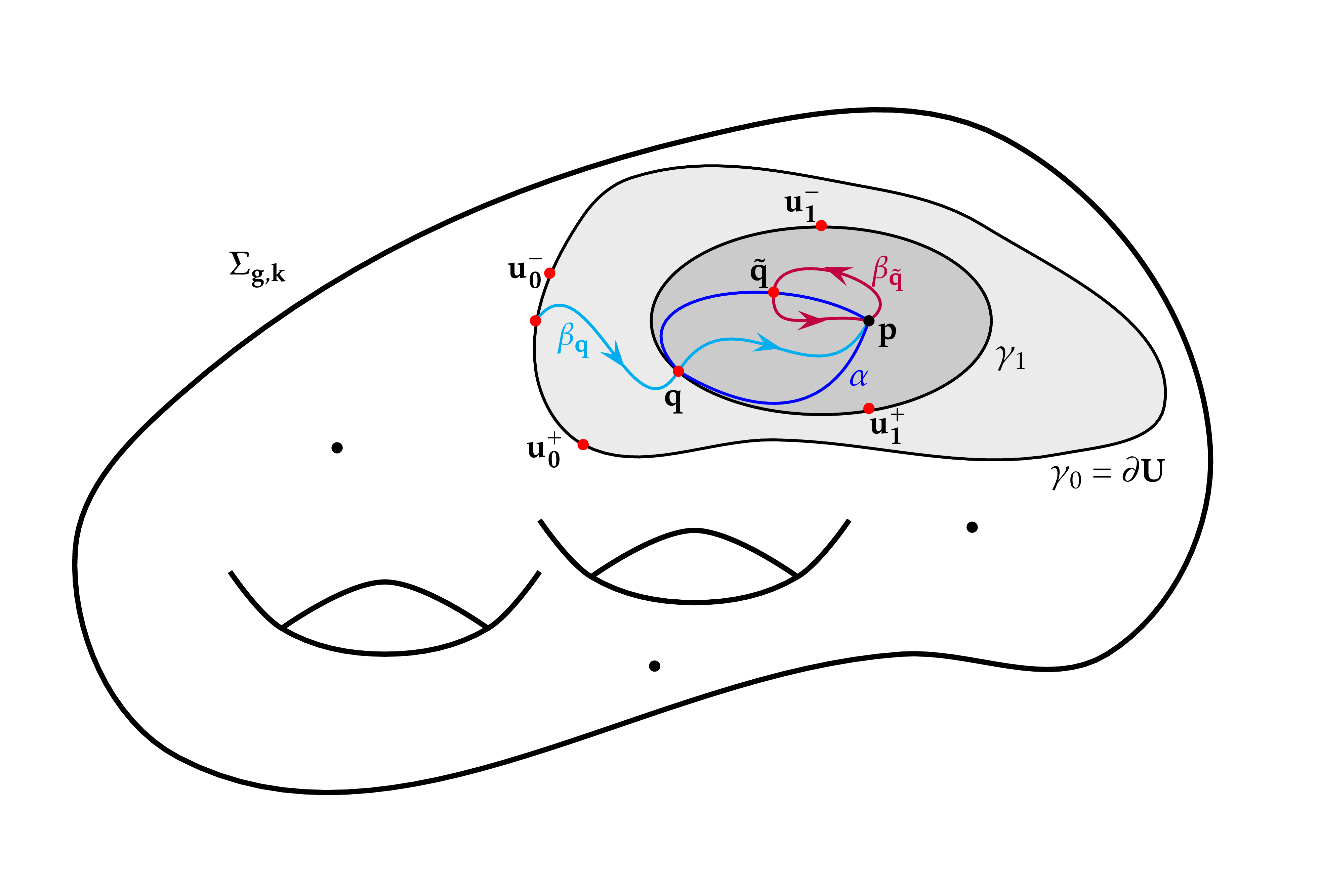}\caption{The level set $\alpha$ cannot tend to the puncture $p$, if it contains a point $q$ that is too far away from $\gamma_0$.}\label{puncture}
\end{figure}
Hence $\alpha$ is a simple closed level curve enclosing $p$. Now let 
$V:=\operatorname{int}(\gamma_1)$, this proves the claim.

{\bf Step 2.} We will use the above mentioned results to proceed with the gluing of spherical caps along the ends
of our surface. For each puncture $p_j\in\{p_1,\dots,p_k\}$ choose a closed level curve $\alpha_j$ and an
open set $V_j$ as in Step 1. From Morse Theory (see for example \cite{m}) it follows that all level curves $\tilde\alpha_j$
contained in $V_j$ are isotopic. Since for each puncture exactly one of the conditions
$$\limsup_{x\to p_j}u(x)=\infty\quad\text{ or}\quad \liminf_{x\to p_j}u(x)=-\infty$$
holds, we can without loss of generality assume that 
$\lim_{x\to p_j}u(x)=\infty$
for any index $j\in\{1,\dots,l\}$ and
$\liminf_{x\to p_j}u(x)=-\infty$
for any index $j\in\{l+1,\dots,k\}$, where $l\in\{0,\dots,k\}$, and that
$$\bigcup_{1\le j\le k}\alpha_j=|u|^{-1}(L)$$
for a large positive constant $L$. Let $E\subset\real{3}$ be the $2$-dimensional subspace perpendicular
to $\v$ and denote by
$$P_{\pm L}:=E\pm L\v$$
the affine planes parallel to $E$ at distance $L$. Then
\begin{equation}\label{sort}
\bigcup_{1\le j\le l}\alpha_j\subset P_{L}\quad\text{and}\quad\bigcup_{l+1\le j\le k}\alpha_j\subset P_{-L}.
\end{equation}
For each curve $\alpha_1,\dots,\alpha_l$ let us define
$$a_j:=\max_{p\in\alpha_j}|p-L\v|$$
and for the curves $\alpha_{l+1},\dots,\alpha_k$ we set
$$b_j:=\max_{p\in\alpha_j}|p+L\v|.$$
It is then even possible to sort the curves in such a way that $\alpha_l$ is the outermost and $\alpha_1$ the innermost curve in $P_L$, measured from the point $L\v$, i.e. so that
$a_1\le \dots \le a_l$.
In the same way one can sort
the curves $\alpha_{l+1}\le\dots\le\alpha_k$, so that $b_{l+1}\le\dots\le b_k$.
If $f$ is an embedding, the curves $\alpha_j$ do not intersect each other.
This implies that the interior of the curve $\alpha_j$
cannot contain any of the curves $\alpha_{j'}$ for $1\le j<j'\le l$ and for $l+1\le j<j'\le k$.
Using Morse Theory again we can replace $(\alpha_j,V_j)$ by the level curve $(\tilde\alpha_j,\tilde V_j)$
for the value $L+j-1$,
$j=1,\dots,l$ and by the value $-L-j+(l+1)$ for $l+1\le j\le k$.
By Lemma \ref{spherical} we can now do the following surgery. 
\begin{figure}[h]\includegraphics[scale=.05]{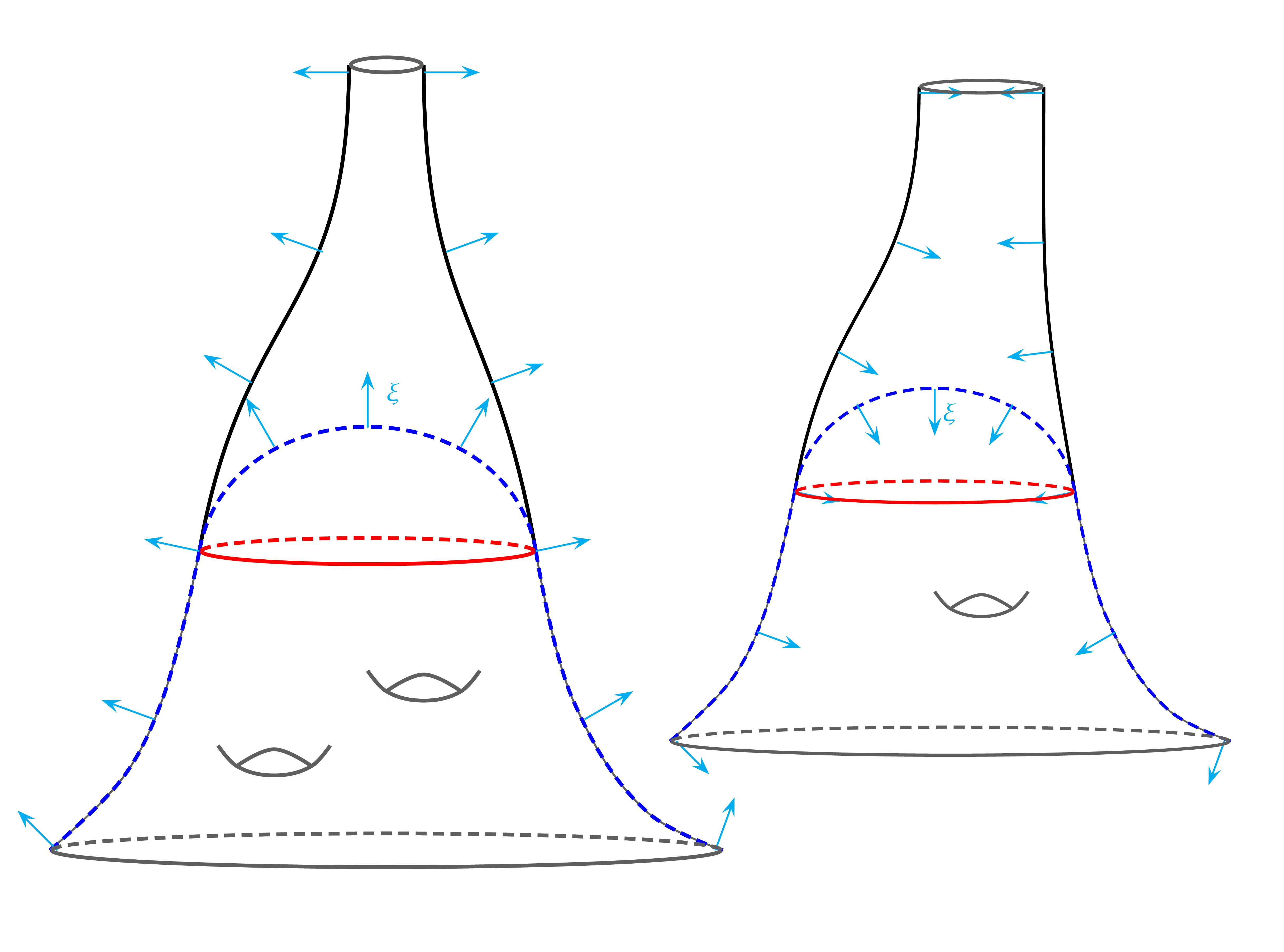}\caption{Gluing spherical caps to truncated ends.}\label{gluing}
\end{figure}
First we remove all ends, that is consider the set
$$M:=f\left(\Sigma_{g,k}\smi\bigcup_{1\le j\le k}\tilde V_j\right).$$
To each end in the upper half space, i.e. for $1\le j\le l$, we smoothly glue an upper spherical cap
to $M$ as described in Lemma \ref{spherical} with some $\sigma>0$. To each end in the lower half space,
i.e. for $l+1\le j\le k$, we smoothly glue a lower spherical cap with the same $\sigma$. Choosing
$\sigma$ sufficiently small we can guarantee that the result is still embedded, if that was the
case for $f$. We obtain a smooth immersion (resp. an embedding if $f$ is one) 
$F:\Sigma_g\to\real{3}$ such that
$$F_{|\left(\Sigma_{g,k}-\bigcup_{1\le j\le k}\tilde V_j\right)}=f.$$
{\bf Step 3}. We use now degree theory to investigate the Gau{\ss} image. Suppose the Gau\ss\ map $\xi$ of $f:\Sigma_{g,k}\to\real{3}$ does not contain $\v$ or
$-\v$. Without loss of generality we assume that $\v$ is not attained, since the case that the Gau\ss\ map does not attain $-\v$ can be treated in the same way.

Denote by $\tilde\xi$ the Gau\ss\ map of $F$. Since $\xi^{-1}(\v)=\emptyset$ we see that the
new Gau\ss\ map $\tilde\xi$ can attain the value $\v$ only at the poles of the added caps.
By Lemma \ref{spherical} the Gau\ss\ curvature at the poles is strictly positive so that
$\v$ is a regular value of $\tilde\xi$. \\[12pt]
It is well known (see for example \cite{hopf}) that the degree of the Gau{\ss} map $\tilde\xi$ of the compact surface $\Sigma_g$ is equal to
$$\operatorname{deg}\tilde\xi=1-g.$$
On the other hand
$$\operatorname{deg}\tilde\xi
=\sum_{{\p}\in\tilde\xi^{-1}({q})}\operatorname{sign}\operatorname{det}\operatorname{d}\hspace{-2pt}{\tilde\xi}({\p})
=\sum_{{\p}\in\tilde\xi^{-1}({q})}\operatorname{sign}K({\p}),$$
where ${q}$ is an arbitrary regular value of $\tilde\xi$ (see for example \cite{milnor}). Take as ${ q}$ the value ${\v}$
of $\mathbb{S}^2$. Note that $\tilde\xi^{-1}({\v})$ consists of points at the poles of the spherical
caps that we added. Thus
\begin{eqnarray*}\label{south}
1-g&=&\{\text{number of poles where }\tilde\xi=\v\}\\
&\ge& 0.
\end{eqnarray*}
Thus, the first assertion of the theorem is true. Let us now investigate the case where $f$ is an embedding. 
Since in case of embedded surfaces the unit normal vector field $\tilde\xi$
can be chosen to be outward pointing, we can move a plane perpendicular to $\v$ from infinity by parallel transport until it touches the
surface from above. So in this case there exists at least (and at most) one pole, where $\tilde\xi=\v$.
Consequently,
\begin{eqnarray*}\label{south2}
1-g&=&\{\text{number of poles where }\tilde\xi=\v\} \\
&\ge& 1,
\end{eqnarray*}
which yields $g=0$. This completes the proof of the theorem.
\end{proof}

\subsection{Translating surfaces}
Suppose $f:M^2\to\real{3}$ is an immersion of an oriented manifold $M$ as a translating surface in direction of $\v$, where $\v$
shall denote the north pole of $\mathbb{S}^2$. Since the Gau\ss\ map $\xi:M^2\to\mathbb{S}^2$ satisfies
$$H=-\langle\xi,\v\rangle,$$
we have $H\in[-1,1]$  and we immediately observe that:
\begin{enumerate}[\rm(a)]
\item The mean curvature $H$ is strictly positive if and only if around each point of $M^2$ the hypersurface is a graph over
a portion of the plane with normal $\v$.
\smallskip
\item The mean curvature $H$ is strictly bigger than $-1$ if and only if the vector $\v$ is
not contained in the Gau{\ss} image of $f$.
\end{enumerate}
Note, that the grim hyperplane and the translating paraboloid are both mean convex, that is they satisfy the above condition (a). Moreover,
the translating catenoid satisfies condition (b), that is its Gau\ss\ map omits the
north pole. In the next theorem we prove that a translating soliton for which the Gau\ss\ map omits the north pole must be even strictly mean
convex, if it is already mean convex outside some compact subset. Before, stating and proving this result, let us give the following useful lemma.

\begin{lemma}\label{meanconvex}
Let $\mathscr{C}$ denote the class of translators
$f:M^2\to\real{3}$ for which exists a constant $\varepsilon>0$ such that
its Gau{\ss} curvature satisfies $K\ge\varepsilon$ on the set
$$\{x\in M^2:H(x)=-1\},$$
where $\{x\in M^2:H(x)=-1\}=\emptyset$ is allowed. If $f\in\mathscr{C}$
satisfies $H\ge 0$ outside a compact subset $C$ of $M^2$, then either $H\equiv 0$
and $M=f(M^2)$ is isometric to a plane or $H>0$ on all of $M^2$. 
\end{lemma}

\begin{proof}
Let $\lambda>0$ be a constant to be chosen later and define the function $g:M^2\to\real{}$ given by
$g:=He^{\lambda u}$. A straightforward computation yields
\begin{eqnarray*}
\nabla g&=&e^{\lambda u}(\nabla H+\lambda H\nabla u),\\
\Delta g&=&e^{\lambda u}\big\{\Delta H+2\lambda\langle\nabla H,\nabla u\rangle+\lambda^2H|\nabla u|^2+\lambda H\Delta u\big\}.
\end{eqnarray*}
Now let $\mu:=\inf_{M^2}g(x)$. We claim that $\mu\ge 0$. Suppose to the contrary that $\mu$ is negative. Since $C$ is compact and $g$ is non-negative
on $M^2\smi\, C$, the infimum is attained at some point $x_0\in C$. Thus, $H(x_0)<0$. Moreover at the point $x_0$ we have $\nabla g=0$ and $\Delta g\ge 0$,
that is
$$\nabla H=-\lambda H\nabla u$$
and
\begin{eqnarray}\label{delta}
0&\le&\Delta H+2\lambda\langle\nabla H,\nabla u\rangle+\lambda^2H|\nabla u|^2+\lambda H\Delta u\nonumber\\
&=&-H^3+2HK-\lambda(2\lambda-1)H|\nabla u|^2+\lambda^2H|\nabla u|^2+\lambda H^3\nonumber\\
&=&-(1-\lambda)H^3+\lambda(1-\lambda)H|\nabla u|^2+2HK\nonumber\\
&=&(1-\lambda)H\big\{\lambda|\nabla u|^2-H^2\big\}+2HK.
\end{eqnarray}
On the other hand, since always
$$K|\nabla u|^2=-|\nabla H|^2-H\langle \nabla H,\nabla u\rangle,$$
we get that at the point $x_0$ it holds
$$K|\nabla u|^2=-\lambda^2H^2|\nabla u|^2+\lambda H^2|\nabla u|^2=\lambda(1-\lambda)H^2|\nabla u|^2.$$
We distinguish now two cases:

{\bf Case 1:} Suppose that $|\nabla u|(x_0)>0$.
Choose the parameter $\lambda$ to take values in the interval $[1/2,1)$. Then the Gau{\ss} curvature $K$ must be positive
at $x_0$ and in particular
$$K=\lambda(1-\lambda)H^2.$$
Substituting the above expression of $K$ into the inequality (\ref{delta}) we deduce that
\begin{eqnarray*}
0&\le&(1-\lambda)H(-H^2+\lambda|\nabla u|^2)+2HK\\
&=&(1-\lambda)H\big\{(2\lambda-1)H^2+\lambda|\nabla u|^2\big\}\\
&<&0
\end{eqnarray*}
which leads to a contradiction. Consequently, $\mu\ge 0$ and thus $H\ge 0$ on all of $M$. But then the strong elliptic maximum principle
applied to Lemma \ref{lemm rel} (f) gives either $H\equiv 0$ or $H>0$ on  $M$.

{\bf Case 2:} Suppose now that $|\nabla u|(x_0)=0$. Then, at the point $x_0$ we have
$H=-1$. From the inequality (\ref{delta}) we obtain that at the point $x_0$, the following inequality
is valid. 
$$\lambda-1+2K\le 0.$$
Using the fact that $K\ge\varepsilon$ at $x_0$ we deduce that
\begin{equation}\label{delta2}
\lambda-1+2\varepsilon\le 0.
\end{equation}
On the other hand, notice that at $x_0$ we have
$$\varepsilon\le K\le {H^2}/{4}={1}/{4}.$$
Therefore, for $\lambda\in(1-2\varepsilon, 1)$ equation (\ref{delta2}) leads to a contradiction. Thus $H\ge 0$
everywhere and again by the strong maximum principle it follows that either $H\equiv 0$ or $H>0$ everywhere.

It is obvious that it is possible to choose a parameter $\lambda$ which works simultaneously for both cases.
In the case where $M$ is complete and properly embedded it turns out that it must be represented as a
graph over a region of the plane $E$ that is perpendicular to the translating direction $\v$. Consequently, in
this case the genus of $M$ is zero.
\end{proof}

As an immediate consequence of the above lemma we get the following:

\begin{mythm}\label{thm liminf}
Let $f:M^2\to\real{3}$ be a translating soliton whose mean curvature satisfies $H>-1$.
Suppose that $H\ge 0$ outside a compact subset of $M^2$. Then either $H\equiv 0$
and $M=f(M^2)$ is isometric to a plane or $H>0$ on all of $M^2$. If, additionally, $M$ is properly
embedded then it is a graph and so it has genus zero.
\end{mythm}

Following the same strategy as in the proof of Lemma \ref{meanconvex} we can show the following:

\begin{lemma}
Let $f:M^2\to\real{3}$ be a translating soliton of the mean curvature flow. Suppose that
\begin{enumerate}[(a)]
\item
there exists a positive constant $\varepsilon$ such that
the Gau{\ss} curvature of $f$ satisfies $K\ge\varepsilon$ on the set
$\{x\in M^2:H(x)=-1\},$
\smallskip
\item
there exists a constant $\lambda\in (1-2\varepsilon,1)$ such that the function $e^{2\lambda u}H^2$ is bounded.
\smallskip
\item
the sub-level sets $M_c:=\{x\in M^2:u(x)\le c\}$, $c\in\real{}$, are compact.
\end{enumerate}
Then, $H>0$ on all of $M^2$.
\end{lemma}
\begin{proof}
Choose a positive constant $\delta>0$ such that
$$\lambda-\frac{\delta}{2}>1-2\varepsilon.$$
Because by assumption the function $H^2e^{2\lambda u}$ is bounded, we deduce that
there exists a positive constant $C$ such that the following estimate
$$H^2e^{(2\lambda-\delta)u}\le Ce^{-\delta u}$$
holds. Because the height $u$ is unbounded from above we get that
$H^2e^{(2\lambda-\delta)u}$ tends to $0$ as $u$ goes to infinity.
Consider the function $g:M\to\real{}$ given by
$$g:=He^{(\lambda-\delta/2)u}.$$
Suppose that there exists a point $x\in M^2$ where $H< 0$. Proceeding as in the proof of Lemma \ref{meanconvex} we get a contradiction.
Thus, either $H\equiv 0$ and $f(M^2)$ is a plane or $H>0$. The first alternative is impossible because of assumption (c). Hence, $H$ must be positive everywhere.
\end{proof}

\begin{mythm}
Let $f:\Sigma_{g,k}=\Sigma_g\smi\{p_1,\dots,p_k\}\to\real{3}$ be a complete translating soliton that satisfies the following two conditions:
\begin{enumerate}[\rm(a)]
\item each end is either bounded from above or from below,
\smallskip
\item the mean curvature satisfies $H>-1$ everywhere and
$$\limsup_{x\to p_j}H^2(x)\le 1-\varepsilon$$
for some positive constant number $\varepsilon$ and for
any of the punctures $p_j\in\{p_1,\dots,p_k\}$. 
\end{enumerate}
Then $g\le 1$. If $f$ is an embedding, then $g=0$ and thus $\Sigma_{g,k}$ is a planar domain.
\end{mythm}
\begin{proof}
Recall that the height function $u=\langle f,\v\rangle$ satisfies the equation
$$|\nabla u|^2+H^2=1.$$
Since the mean curvature satisfies  $H>-1$, we deduce that $\v$ does not belong to the Gau{\ss} image of $f$. Hence the result follows
as a corollary of Theorem \ref{gauss map}.
\end{proof}

We will conclude this section with some interesting formulas relating the Gau\ss\ curvature $K$ of a translating soliton in $\real{3}$ with its
mean curvature $H$.
\begin{lemma}\label{lemm hk}
On a translating soliton in $\real{3}$ the Gau{\ss} curvature $K$ satisfies
the following equations to which we will refer to as the $(H,K)$-formulas. 
\begin{eqnarray}
K&=&\Delta\log\sqrt{1+H^2}-\frac{2}{(1+H^2)^2}|\nabla H|^2+\frac{H^4}{1+H^2}\label{hk1}\\
&=&\frac{H}{1+H^2}\Delta H-\frac{|\nabla H|^2}{1+H^2}+\frac{H^4}{1+H^2}\label{hk2}
\end{eqnarray}
Moreover, at each point where $H> 0$ we have
\begin{equation}\label{hk3}
K=\frac{H^2}{1+H^2}\Delta\log(e^uH)=\frac{H^2}{1+H^2}\left(\Delta\log H+H^2\right).
\end{equation}
\end{lemma}
\begin{proof}
From Lemma \ref{lemm rel} (g), we have
\begin{equation}\label{ricci}
K|\nabla u|^2=\operatorname{Ric}(\nabla u,\nabla u)=-|\nabla H|^2-H\langle\nabla H,\nabla u\rangle.
\end{equation}
Taking into account the formula for $\Delta H$ in Lemma \ref{lemm rel} (f) we compute
\begin{eqnarray*}
\Delta\log\sqrt{1+H^2}&=&\frac{H}{1+H^2}\Delta H+\frac{1-H^2}{(1+H^2)^2}|\nabla H|^2\\
&=&-\,\frac{H}{1+H^2}\left(H|h|^2+\langle\nabla H,\nabla u\rangle\right)\\
&&+\frac{1-H^2}{(1+H^2)^2}|\nabla H|^2.
\end{eqnarray*}
Combining the last equality with (\ref{ricci}) and $|h|^2=H^2-2K$ we get
\begin{eqnarray*}
\Delta\log\sqrt{1+H^2}&=&-\frac{H^4}{1+H^2}+2K\frac{H^2}{1+H^2}\\
&&+\frac{1}{1+H^2}\left(K|\nabla u|^2+|\nabla H|^2\right)\\
&&+\frac{1-H^2}{(1+H^2)^2}|\nabla H|^2\\
&=&-\frac{H^4}{1+H^2}+K\frac{2H^2+|\nabla u|^2}{1+H^2}\\
&&+\frac{2}{(1+H^2)^2}|\nabla H|^2\\
&=&-\frac{H^4}{1+H^2}+K+\frac{2}{(1+H^2)^2}|\nabla H|^2.
\end{eqnarray*}
This completes the proof of (\ref{hk1}). Then (\ref{hk2}) is just a reformulation of (\ref{hk1}).
Finally, if $H>0$, then $\log H$ is well defined and (\ref{hk3}) follows from $\Delta u=H^2$ and
(\ref{hk2}).
\end{proof}

\begin{corollary}
Outside of the critical set of $u$, the vector field
$$W=-\frac{\nabla H+H\nabla u}{|\nabla u|^2}=-\frac{\nabla(e^uH)}{e^u|\nabla u|^2}$$
is divergence free. Moreover,
$$K=\langle\nabla H,W\rangle=\operatorname{div}(HW).$$
\end{corollary}
\begin{remark}
Note that for the grim hyperplane with $\min u=0$ we have $e^uH=1$.
\end{remark}

\end{document}